\documentclass[11pt,reqno,a4paper]{article}
\usepackage{amsfonts,amssymb,amsmath,amsthm}
\allowdisplaybreaks
\usepackage{mathrsfs}
\usepackage{graphicx}
\usepackage{caption}
\usepackage{enumerate}
\usepackage{fullpage}
\usepackage{multicol}
\usepackage{titling}
\usepackage{authblk}

\newtheorem{theorem}{Theorem}[section]

\newtheorem{lemma}[theorem]{Lemma}
\newtheorem{proposition}[theorem]{Proposition}

\newtheorem{hypothesis}{Hypothesis}

\numberwithin{equation}{section}

\title{Individual variability in dispersal and invasion speed}
\author[a,b]{Aled Morris}
\author[a,c]{Luca B\"orger}
\author[a,b]{Elaine Crooks}
\affil[a]{Centre for Biomathematics, College of Science, Swansea University}
\affil[b]{Department of Mathematics, College of Science, Swansea University}
\affil[c]{Department of Biosciences, College of Science, Swansea University}
\date{}
\begin{document}

\maketitle

\begin{abstract}
We model the growth, dispersal and mutation of two phenotypes of a species using reaction-diffusion equations, focusing on the biologically realistic case of small mutation rates. After verifying that the addition of a small linear mutation rate to a Lotka-Volterra system limits it to only two steady states in the case of weak competition, an unstable extinction state and a stable coexistence state, we prove that under some biologically reasonable condition on parameters the spreading speed of the system is linearly determinate. Using this result we show that the spreading speed is a non-increasing function of the mutation rate and hence that greater mixing between phenotypes leads to slower propagation. Finally, we determine the ratio at which the phenotypes occur at the leading edge in the limit of vanishing mutation.
\end{abstract}

\noindent \begin{amsSC}35C07, 35K57, 92D25\end{amsSC}\medskip

\noindent \begin{keywords}Invasive Species, Linear Determinacy, Population Growth, Mutation, Spreading Speeds, Traveling Waves\end{keywords}

\section{Introduction}

The speed at which a species expands its range is a fundamental parameter in ecology, evolution and conservation biology. A knowledge of this speed enables us to predict the ability of a species to keep up with the rate at which the climate changes or the rate at which an exotic species invades, two prominent ecological challenges \cite{parmesan, sakai}. It is known that traits such as dispersal and population growth affect the rate at which species expands its range, and there has been an intriguing suggestion from recent work that polymorphism in traits can cause a species invasion to occur at a faster rate than a single morph would in isolation \cite{elliott,elliott2}. Understanding the effect that each of a species traits has, and can potentially have, on its rate of spread is therefore important to understanding how the spread of a species can evolve.

Most common models of population dynamics incorporate aspects of dispersal and growth, \textit{e.g.} \cite{lewis, lewis2, murray, turchin}, however the mutation of one phenotype to another is a less common inclusion. Even the addition of a simple mutation term can dramatically affect the behaviour of a model. Cosner \cite{cosner} singles out two models in particular: the model of Elliott and Cornell \cite{elliott} in which a simple linear mutation is used, and the model of Bouin et al \cite{bouin} in which mutations are considered to act as a diffusion process in the phenotype space. Cosner notes that during their work Elliott and Cornell assume without proof that the spread rate of the two phenotypes in their system is determined by the linearisation of their system at the extinction state zero, \textit{i.e.}, that it is linearly determinate. We here prove this assumption in the case when the mutation rate is small, which is the case generally for all organisms, as natural selection typically acts to minimize mutation rate \cite{lynch}. Having shown this result and also established that the spreading speed is equal to the minimal speed of a class of travelling waves, we then prove further results regarding the behaviour of minimal speed travelling waves as a function of mutation rate when the mutation rate is small, and the composition of the leading edge of minimal speed travelling waves in the small mutation-rate limit.

Elliott and Cornell's model examines the interaction between an establisher phenotype with population density $n_e$, and a disperser phenotype with population density $n_d$, using a Lotka-Volterra competition system
\begin{equation}\label{elliott}
	\begin{aligned}	
		\frac{\partial n_e}{\partial t} &= D_e \frac{\partial^2 n_e}{\partial x^2}+r_e n_e(1-m_{ee} n_e - m_{ed} n_d) - \mu_e n_e + \mu_d n_d \\
		\frac{\partial n_d}{\partial t} &= D_d \frac{\partial^2 n_d}{\partial x^2}+r_d n_d(1-m_{de} n_e - m_{dd} n_d) - \mu_d n_d + \mu_e n_e.
	\end{aligned}
\end{equation}
The first term on the right hand side of each equation represents the dispersal of the phenotype through diffusion, where $D_e$ and $D_d$ are the dispersal rates of each morph. The second term describes the growth rate of the phenotype using a logistic term, this is similar to what is used in Fisher's model \cite{fisher}. We use $r_e$ and $r_d$ to represent the growth rate of each morph, $m_{ee}$ and $m_{dd}$ represent the intra-morph competition, while $m_{ed}$ and $m_{de}$ represent the inter-morph competition. The third and fourth terms represent a linear mutation between the phenotypes at mutation rates of $\mu_e$ and $\mu_d$. It is assumed that all parameters of the system are positive real numbers.

Following classical ecological theory on competition–dispersal tradeoffs \cite{lei,tilman}, and as in \cite{elliott}, we suppose that the establisher phenotype has the larger growth rate, while the disperser phenotype has the larger dispersal rate,
\begin{equation}\label{parmsrD}
	r_e > r_d, \,\,\,\, D_d > D_e.
\end{equation}
While \eqref{parmsrD} is not required in the proof of linear determinacy, we will make use of \eqref{parmsrD} in Section 4 where we give some further results on the system \eqref{elliott}. Following classical competition theory \cite{tilman2}, it is also supposed throughout that the intra-morph competition is greater than the inter-morph competition,
\begin{equation}\label{compe}
	m_{dd} > m_{ed}, \,\,\,\,  m_{ee} > m_{de},
\end{equation}
and we will also make further assumptions on the smallness of $m_{de}$ and $m_{ed}$ when necessary. Note that Roques et al \cite{roques} prove for a competition-diffusion model without mutation in which one species invade the territory of another, that linear determinacy only holds for sufficiently small inter-species competition. We have in mind throughout that the mutation is relatively small in comparison to the other parameters, to remain biologically realistic \cite{lynch}, and will again make more precise assumptions when necessary.

Kolmogorov, Petrovskii and Piskunov \cite{kpp} studied the existence of monotonic travelling wave solutions of the scalar form of the equation
\begin{equation}\label{general}
	\frac{\partial u}{\partial t} = A \frac{\partial^2 u}{\partial x^2} + f(u).
\end{equation}
Throughout this work we will consider travelling wave solutions to be  solutions of the equation \eqref{general} of the form $u(x,t) = w(x-ct)$, where $w:\mathbb{R}^n \to \mathbb{R}^n$ is called the wave profile and $c \in \mathbb{R}$ is the speed of the wave. Kolmogorov, Petrovskii and Piskunov studied the case when $n=1$, $A=d$ and $f(u)=ru(1-u)$, proposed by Fisher in \cite{fisher}, and proved there is a continuum of values of $c$ for which a monotonic travelling wave solution exists, specifically if $c > c^*$, where $c^*=2\sqrt{rd}$ is the minimal travelling wave speed. Aronson and Weinberger \cite{aronson} further studied this system and characterised $c^*$ as a spreading speed. These results were extended to cooperative systems of equations for a suitable class of nonlinearities $f$ by Volpert, Volpert and Volpert \cite{volpert3}. Note here that an interesting recent paper of Griette and Raoul \cite{griette} studies the existence and properties of travelling waves for a special case of system \eqref{elliott} in which, in particular, $D_e = D_d$ and $\mu_e = \mu_d$, including an explicit expression for the minimal travelling wave speed given by the linearisation about the extinction state.

The system studied in \cite{elliott} is certainly of the form \eqref{general} and so we let $u=(n_e,n_d)^T \in \mathbb{R}^2$, $A$ be a diagonal matrix containing the dispersal rates,
\begin{equation}\label{A}
	A = 
	\left( 
	\begin{array}{c c}
		D_e & 0\\
		0 & D_d
	\end{array} 
	\right),
\end{equation}
and $f$ a nonlinear function containing the growth, competition and mutation terms,
\begin{equation}
	f(n_e,n_d) = 
	\left( 
	\begin{array}{c}
		r_e n_e (1 - m_{ee} n_e - m_{ed} n_d) - \mu_e n_e + \mu_d n_d \\
		r_d n_d (1 - m_{de} n_e - m_{dd} n_d) + \mu_e n_e - \mu_d n_d
	\end{array} 
	\right).\label{f}
\end{equation}
Throughout this work we will use the notation $u > v$ to denote that the $i$th component of each vector satisfies $u_i > v_i$, for each $i$, similarly for $u<v$, $u \geq v$ and $u \leq v$. We say that $u \in \mathbb{R}^n$ is positive if $u>0$. The notation $u \in (a,b]$ denotes that the $i$th component of each vector satisfies the inequality $a_i < u_i \leq b_i$ for each $i$. Let $C^k(\mathbb{R}^n,\mathbb{R}^n)$ be the set of all continuous functions from $\mathbb{R}^n$ to $\mathbb{R}^n$ such that the $k$th partial derivative of each component exists and is continuous, and let $\mathcal{C}_{r}$ be the set $\mathcal{C}_{r}:= \{ u_i \in C(\mathbb{R},\mathbb{R}), 0\leq u_i \leq r_i, \,\, \mathrm{for} \,\, x \in \mathbb{R}, \,\, i = 1,\ldots,N \}$.

\begin{figure}[ht]
\begin{center}
\includegraphics[scale=1.0]{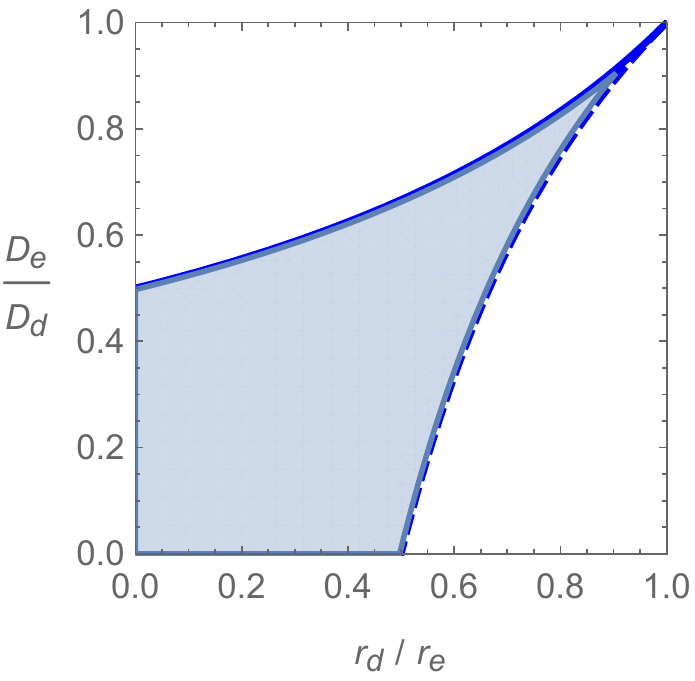}
\end{center}
\caption{Parameter regions in which the faster invasion speed observed by \cite{elliott} occurs. In the upper left unshaded region the solution travels at the speed at which the establisher travels without competition. In the lower right unshaded region the solution travels at the speed at which the disperser travels without competition. In the lower left shaded region the faster spreading speed is observed.}
\label{parameters}
\end{figure}

In \cite{elliott}, Elliott and Cornell investigated numerically the effect of varying the parameters on the spreading speed of the system and interestingly found that, for certain values of growth and dispersal rate, the system would spread faster in the presence of both phenotypes than just one phenotype would spread in the absence of mutation. They predict the spreading speed obtained for each set of parameters in the limit of small mutation, using the front propagation method of van Saarloos \cite{van}, making the assumption that the spreading speed of system \eqref{elliott} is linearly determinate in order to do so. As $\mu \to 0$ the three possible limiting speeds are
\begin{equation}\label{vfspeed}
v_e = 2\sqrt{r_e D_e},\,\,\,\,\,\,\,\, v_d = 2\sqrt{r_d D_d},\,\,\,\,\,\,\,\,v_f = \frac{\vert r_e D_d - r_d D_e \vert}{\sqrt{(r_e - r_d)(D_d - D_e)}}.
\end{equation}
Condition \eqref{parmsrD} is enough to ensure that $v_f$ exists and is faster than $v_e$ and $v_d$, which are the spreading speeds of the two Fisher equations satisfied by each phenotype in isolation. The faster speed $v_f$ is obtained in the region of the positive quadrant of $\left(r_d/r_e,\, D_e/D_d \right)$-space which satisfies the inequalities
\begin{equation}
	\frac{D_d}{D_e} + \frac{r_d}{r_e} > 2, \,\,\,\, \frac{D_e}{D_d} + \frac{r_e}{r_d} > 2,\label{fastersp}
\end{equation}
represented by the shaded area in Figure \ref{parameters}.

The spreading speed of the system is said to be linearly determinate if it is the same as the spreading speed of the system obtained when \eqref{elliott} is linearised about the $(0,0)$ equilibrium, namely
\begin{equation}\label{lin0}
	\begin{aligned}
		\frac{\partial n_e}{\partial t} &= D_e \frac{\partial^2 n_e}{\partial x^2}+(r_e - \mu_e) n_e + \mu_d n_d \\
		\frac{\partial n_d}{\partial t} &= D_d \frac{\partial^2 n_d}{\partial x^2}+(r_d  - \mu_d) n_d + \mu_e n_e.
	\end{aligned}
\end{equation}
This assumption is one that is suggested by the numerical studies in \cite{elliott}, but is not always true even in the scalar case \cite{benguria, hadeler, lucia, van}. Stokes \cite{stokes} calls the minimal wave speed $c^*$ ``pulled'' if it is equal to the linearised spreading speed, that is the speed of the front is determined by the individuals at the leading edge. Similarly the minimal wave speed is said to be ``pushed'' if its speed is greater than the linearised spreading speed, in this case the speed is determined by individuals behind the leading edge. Typically there are qualitative differences in wave behaviour depending on whether the wave is pushed or pulled, e.g. stability in the scalar case is discussed in \cite{rothe}.

In the case of systems, most sufficient conditions for linear determinacy require a cooperative assumption on the system. A system is cooperative when the off-diagonal elements of the Jacobian matrix of $f$ are always non-negative, \textit{i.e.},
\begin{equation}\label{copdef}
	\frac{\partial f_i (u)}{\partial u_j} \geq 0, \,\,\,\, \mathrm{if}\,\, i \neq j.
\end{equation}
In biological terms this would mean that each phenotype benefits from the presence of others. A cooperative system is useful mainly due to the existence of a comparison principle for such systems \cite{lui, weinberger}. The comparison principle is a key tool that can be used in the proof of linear determinacy.

\begin{theorem}[Comparison principle, {\cite[Theorem 3.1]{wang}}]\label{comparisonPrin}
	Let $A$ be a positive definite diagonal matrix. Assume that $f$ is a vector-valued  function in $\mathbb{R}^n$ that is continuous and piecewise continuously differentiable in $\mathbb{R}$, and that the underlying system \eqref{general} is cooperative. Suppose that $u(x,t)$ and $v(x,t)$ are bounded on $\mathbb{R} \times [0,\infty)$, and satisfy
	$$\frac{\partial u}{\partial t} - A \frac{\partial^2 u}{\partial x^2} - f(u) \leq \frac{\partial v}{\partial t} - A \frac{\partial^2 v}{\partial x^2} - f(v)$$
	If $u(x,t_0) \leq v(x,t_0)$ for $x \in \mathbb{R}$, then
	$$u(x,t) \leq v(x,t),  \,\,\,\, \mathrm{for}\,\,x \in \mathbb{R},\,\, t \geq t_0.$$\label{comparison}
\end{theorem}
Linear determinacy was shown to hold for some cooperative systems by Lui \cite{lui}, and the result was extended to more general cooperative systems by Weinberger, Lewis and Li \cite{weinberger}, who assumed, in particular, that for any positive eigenvector $q$ of $f'(0)$,
\begin{equation}\label{wcond}
	f(\alpha q) \leq \alpha f'(0) q, \,\,\,\, \mathrm{for} \,\, \alpha >0.
\end{equation}
This is equivalent to a condition imposed by Hadeler and Rothe \cite{hadeler} in the scalar case. Note that it is clearly sufficient to show that inequality \eqref{wcond} holds for all positive vectors, which implies the inequality holds for positive eigenvectors. Unfortunately we see from the Jacobian matrix that our system \eqref{elliott} is only partially cooperative,
\begin{equation}\label{Jf}
	J_f(n_e, n_d) = 
\left( \begin{array}{cc}
r_e(1-2 m_{ee} n_e - m_{ed} n_d)-\mu_e &  \mu_d - r_e m_{ed} n_e\\
\mu_e - r_d m_{de} n_d & r_d(1 - m_{de} n_e -2 m_{dd} n_d)-\mu_d
\end{array} \right).
\end{equation}
In fact it is only cooperative at small population densities due to the relative smallness of the mutation term.

After establishing linear determinacy, we turn to some ecologically important questions in the case of small mutation rate. We first exploit the Perron-Frobenius structure of $H_\beta$ in order to investigate the effect of mutation on spreading speed, and show that an increase in mutation between morphs results in a decrease in the value of the minimal spreading speed $c^*$. Secondly we investigate the composition of the leading edge of invasion in the limit of small mutation rate, and demonstrate the effects of dispersal, growth rate and mutation on this composition.

A framework for showing linear determinacy for certain partially cooperative systems was introduced by Haiyan Wang \cite{wang}. In order to use the main theorem of \cite{wang}, the following three hypotheses are required to hold. Hypothesis 1 centres around finding functions $f^+$ and $f^-$ which bound \eqref{f} from above and below respectively.

\begin{hypothesis}
Let $f \in C^2(\mathbb{R}^n,\mathbb{R}^n)$ be such that
\begin{enumerate}[{\normalfont (i)}]\itemsep0em
\item $f(0) = f(k) = 0$ were $k \in \mathbb{R}^n$ is a positive constant known as the coexistence equilibrium. There are no other positive equilibria of $f$ between 0 and $k$. Further, $f$ has a finite number of equilibria.
\item There exist functions $f^+$ and $f^-$, with respective positive equilibria $k^+$ and $k^-$, such that
$$k^- \leq k \leq k^+,$$
and for $u \in [0,k^+] $
$$f^-(u) \leq f(u) \leq f^+(u).$$
There are no other positive equilibria of $f^\pm$ between 0 and $k^\pm$.
\item The set $\mathcal{C}_{k^+}$ is an invariant set, in the sense that for any $u_0 \in \mathcal{C}_{k^+}$, the solution of \eqref{elliott} with initial condition $u_0$ exists and remains in $\mathcal{C}_{k^+}$ for all $t \in [0,\infty)$.
\item $f^+$ and $f^-$ are cooperative in the sense that they satisfy \eqref{copdef} for $u \in [0,k^{\pm}]$.
\item $f^\pm(u)$ have the same Jacobian matrix as $f(u)$ at $u=0$.
\end{enumerate}
\end{hypothesis}

If \eqref{general} admits a travelling wave solution $u(x,t)=w(x-ct)$, then substituting $w(x-ct)$ in to \eqref{elliott} we obtain the travelling wave equation
\begin{equation}\label{ode2o}
	A w''(\xi) + c w'(\xi) +f(w(\xi)) = 0,
\end{equation}
the linearisation of which about the origin is,
\begin{equation}\label{chap1lin}
	A w''(\xi) + c w'(\xi) +f'(0)w(\xi) = 0.
\end{equation}
Further, if we substitute the ansatz solution $w(\xi) = e^{- \beta \xi} q$ into \eqref{chap1lin}, we obtain the eigenvalue problem
\begin{equation}\label{eig22}
	\frac{1}{\beta}(\beta^2 A + f'(0))q = c q,
\end{equation}
where $\beta > 0$ is the spatial decay and $q > 0$ denotes the phenotypic distribution at the leading edge. We define
\begin{equation}\label{hlambda}
	H_{\beta} := \beta A + \beta^{-1}f'(0).
\end{equation}
The existence of positive $\beta$, $q$ satisfying \eqref{eig22} is a necessary but not sufficient condition for the existence of a solution $w$ of \eqref{ode2o} with $w>0$ and $w(\xi) \rightarrow 0$ as $\xi \rightarrow \infty$ {{\cite[page 163, Lemma 2.4]{volpert3}}}.

\begin{theorem}[Perron-Frobenius, {\cite{horn}}]\label{PerronF}
	Let the $n \times n$ matrix $P$ be irreducible and have non-negative off-diagonal elements, and suppose that $n \geq 2$. Then $P$ has a real, simple eigenvalue $\eta_\mathrm{PF}(P)$ such that an associated eigenvector is positive and every other eigenvalue of $P$ has real part less than $\eta_\mathrm{PF}(P)$. Moreover, any positive eigenvector must be a multiple of the eigenvector corresponding to $\eta_\mathrm{PF}(P)$. We call $\eta_\mathrm{PF}(P)$ the Perron-Frobenius eigenvalue of $P$ and the associated positive eigenvector $q$ with norm $\Vert q \Vert = 1$ the Perron-Frobenius eigenvector of $P$.
\end{theorem}

\begin{hypothesis}
	The matrix $H_{\beta}$ is in block lower triangular form, with each diagonal block irreducible, has non-negative off-diagonal elements, and the first diagonal block has a positive Perron-Frobenius eigenvalue $\eta$ which is strictly larger than the Perron-Frobenius eigenvalues of all other irreducible diagonal blocks for $\beta>0$. There is also a corresponding positive eigenvector $q > 0$ which is continuous with respect to $\beta$ for $\beta>0$.
\end{hypothesis}

\noindent In our case, $H_\beta$ only consists of one irreducible diagonal block and has positive off-diagonal elements, so we know that $H_\beta$ itself has a Perron-Frobenius eigenvalue, which we call $\eta_\beta$, with corresponding eigenvector $q_\beta$. We thus rewrite \eqref{eig22} in the form
\begin{equation}\label{eigeta}
	H_\beta \, q_\beta = \eta_\beta \, q_\beta.
\end{equation}
The Perron-Frobenius eigenvalue $\eta_\beta$ is a function of $\beta$ where each value $\eta_\beta$ is a possible spreading speed of a decreasing travelling-wave solution of the linearised problem \eqref{lin0}. We define
\begin{equation}\label{cstar}
	c^* = \inf_{\beta>0} \eta_\beta .
\end{equation}
Note that the $c^*$ we define here is the spreading speed for the linear problem \eqref{lin0}.

\begin{hypothesis}
	For any $\alpha > 0$, $\beta > 0$
\begin{equation}\label{h3}
	f^\pm (\alpha q_\beta ) \leq \alpha f'(0) q_\beta,
\end{equation}
where $q_\beta$ is the Perron-Frobenius eigenvector of $H_\beta$.
\end{hypothesis}

\noindent We now introduce the following theorem from \cite{wang}, which is concerned both with spreading speeds and the existence of travelling waves.

\begin{theorem}[{{\cite[Theorem 2.1]{wang}}}]\label{wang}
If  \eqref{elliott} satisfies Hypotheses 1-3 then:
\begin{enumerate}[{\normalfont (i)}]\itemsep0em
\item For any $u_0 \in \mathcal{C}_{k^+}$ with compact support, the solution $u(x,t)$ of \eqref{general} with initial condition $u_0$ satisfies
$$\lim_{t \rightarrow \infty} \sup_{\vert x \vert \geq c t} u(x,t) = 0 \mathrm{, \,\,\,\, for}\,\,  c>c^*\mathrm{,}$$
where $c^*$  is defined in \eqref{cstar}.
\item For any vector $\omega \in \mathbb{R}^N$, $\omega > 0$, there is a positive $R_\omega$ with the property that if $u_0 \in \mathcal{C}_{k^+}$ and $u_0 \geq \omega$ on an interval of length $2R_{\omega}$, then $u(x,t)$ with initial condition $u_0$ satisfies
$$0<k^- \leq \liminf_{t \rightarrow \infty}\inf_{\vert x \vert \leq c t} u(x,t) \leq k^+, \,\,\,\, \mathrm{for}\,\,  0<c<c^*\mathrm{.}$$
\item For each $c>c^*$ there is a travelling wave solution $u(x,t) = w(x-c t)$, such that for $\xi \in \mathbb{R}$, $0 < w(\xi) \leq k^+$,  $\beta >0$ satisfying \eqref{eigeta} with $c=\eta_\beta$, and $q_\beta$ satisfying \eqref{eigeta} for this $\beta$,
$$0 < k^- \leq \liminf_{\xi \rightarrow - \infty} w(\xi) \leq \limsup_{\xi \rightarrow - \infty} w(\xi) \leq k^+,$$
$$\lim_{\xi \rightarrow \infty} w(\xi)\mathrm{e}^{\beta \xi} = q_\beta,$$
where $k^+$ is defined in Hypothesis 1.
\item For $c=c^*$ there is a non-constant travelling wave solution such that $0 \leq w(\xi) \leq k^+$, $\xi \in \mathbb{R}$.
\item For $0<c<c^*$ there is no travelling wave solution with $\liminf_{\xi \rightarrow - \infty}w(\xi) > 0$, and $\lim_{\xi \rightarrow + \infty}w(\xi) = 0$.
\end{enumerate}
\end{theorem}

\noindent Note that we have made some natural modifications of {{\cite[Theorem 2.1]{wang}}} to fit our context of travelling wave solutions that tend to zero at $+ \infty$.

The first two parts of Theorem \ref{wang} tell us of the existence of a spreading speed $c^*$ for the system \eqref{elliott} and that this speed is the spreading speed of the linear problem, given by \eqref{cstar}. The remaining three parts tell us of the existence of travelling wave solutions for $c \geq c^*$ and gives some information about these travelling waves. Part (iii) tells us that travelling wave solutions of non-minimal speed exist, and that these solutions decay exponentially to $0$ as $\xi \to \infty$, and are bounded between $k^-$ and $k^+$ as $\xi \to -\infty$. Part (iv) tells us that travelling wave solutions of minimal speed exist and that they remain bounded between $0$ and $k^+$ for all $\xi \in \mathbb{R}$. Part (v) tells us that there exist no travelling wave solutions for $c < c^*$ which tend to $0$ as $\xi \to +\infty$ and are positive as $\xi \to -\infty$.

Note that part (iv) does not tell us much about the form of minimal-speed waves other than they are bounded between $0$ and $k^+$ for all $\xi \in \mathbb{R}$, however we would expect that these solutions tend to $0$ as $\xi \to \infty$, and $k$ as $\xi \to -\infty$ (and in fact that non-minimal speed waves also tend to $k$ as $\xi \to -\infty$). It has since been shown by Girardin \cite{girardin} that minimal-speed waves do indeed tend to $0$ as $\xi \to \infty$, although the question of their precise behaviour as $\xi \to -\infty$ remains open. Our expectations on the behaviour of minimal-speed solutions are also supported by the numerical simulations of Elliott and Cornell, which we have recreated in Figure \ref{twavesol} using a finite difference method with a Heavyside initial condition for each component. Provided that the long-time behaviour from such initial conditions is indeed as one would expect, close to the minimal-speed travelling wave, the simulations suggest that the form of this wave has a non-monotone component. Note that this is consistent with the work of Griette and Raoul \cite{griette} who consider travelling wave solutions of a similar system in which both phenotypes disperse at the same rate, $D_e = D_d$. Griette and Raoul exploit this fact in order to describe the form of travelling wave solutions of their system, finding a similar non-monotone behaviour to the one seen in Figure \ref{twavesol}. Another natural question is, how does the travelling wave behave near $-\infty$, and in particular, what is the relationship between the co-existence equilibrium and the limit at $- \infty$, if it exists, of travelling wave solutions? These are non-trivial, interesting questions which we do not focus on here and instead leave for future work.

\begin{figure}[ht]
\begin{center}
\includegraphics[scale=0.5]{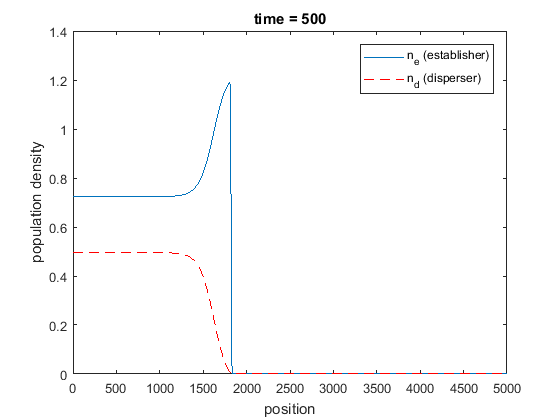}
\end{center}
\caption{Solution of the system \eqref{elliott} with parameter values: $D_e = 0.3$, $D_d = 1.5$, $r_e = 1.1$, $r_d = 0.2$, $m_{ee} = 1.0/1.2$, $m_{dd} = 1.0$, $m_{ed} = 0.8$, $m_{de} = 0.7$,  $\mu_e = 0.001$, $\mu_d = 0.00025$. A Heaviside step function was used as initial condition for each component.}
\label{twavesol}
\end{figure}

Note that the parallel work of Girardin \cite{girardin} develops a general linear determinacy result under certain conditions, which includes the case considered here. Here we present an alternative approach in the more biologically realistic case of small mutation rate, using the intuitive method of trapping $f$ between $f^+$ and $f^-$. In addition to being tools in the proof, $f^-$ and $f^+$ also give us some additional quantitative information on the behaviour of solutions, e.g., our method establishes $k^-$ as a lower bound for solutions at $-\infty$.


In Section 2 we verify that there only exist two non-negative equilibria for the system \eqref{elliott}: an unstable extinction equilibrium $(0,0)$ and a stable coexistence equilibrium $(n_e^*,n_d^*)$. In Section 3 it is proved that, under certain conditions on parameters, the spreading speed of the system is indeed linearly determinate, using Theorem \ref{wang}. In Section 4 we then use this result in order to analyse the dependence of the spreading speed on mutation and the behaviour of the two phenotypes in the leading edge of the travelling wave in the limit of small mutation rate.

\section{Equilibria of the system}

A much studied competition-diffusion system, similar to \eqref{elliott} but where there is no mutation between phenotypes and both intra-morph competition values equal one, is the Lotka-Volterra system of equations \cite{dancer,ni,roques,shigesada},
	\begin{equation}\label{lvlewis}
		\begin{aligned}
			\frac{\partial n_e}{\partial t} &= D_e \frac{\partial^2 n_e}{\partial x^2}+r_e n_e(1- n_e - m_{ed} n_d), \\
			\frac{\partial n_d}{\partial t} &= D_d \frac{\partial^2 n_d}{\partial x^2}+r_d n_d(1- n_d - m_{de} n_e).
		\end{aligned}
	\end{equation}
Lewis, Li and Weinberger \cite{lewis} note that a coexistence equilibrium for this system \eqref{lvlewis} exists if and only if $(1-m_{ed})(1-m_{de})>0$; that is, either when both $m_{ed} < 1$ and $m_{de}<1$, or both $m_{ed} > 1$ and $m_{de}>1$. Note that the case where $m_{ed}<1$ and $m_{de}<1$ corresponds to the condition \eqref{compe} in our system \eqref{elliott}.  A stability analysis shows that this coexistence equilibrium is stable when $m_{ed}<1$ and $m_{de}<1$, and unstable when $m_{ed}>1$ and $m_{de}>1$. It should also be noted that \cite{lewis, roques} study the case where one species invades the territory of another, while we and also \cite{elliott, griette}, are concerned with two morphs of a species invading a previously unoccupied territory.

\begin{figure}[ht]
\begin{center}
\includegraphics[scale=1.0]{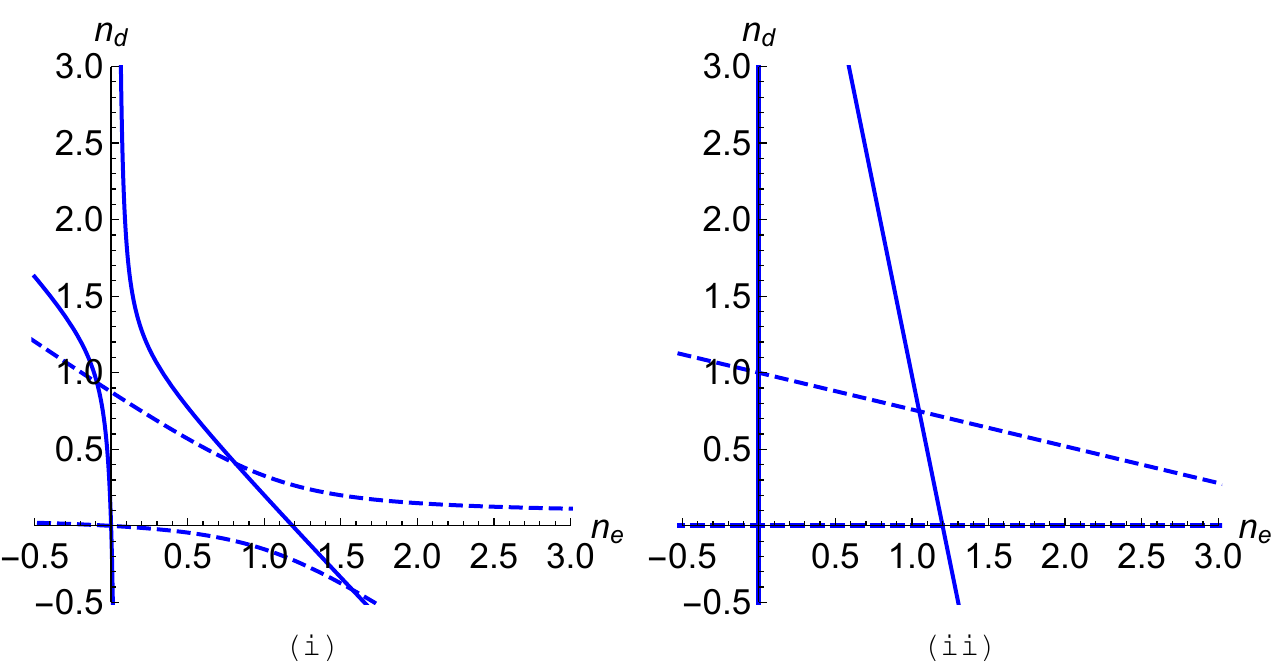}
\end{center}
\caption{(i) Nullclines of equation \eqref{f} and (ii) Nullclines of equation \eqref{g}. Parameter values $D_e = 0.3$, $D_d = 1.5$, $r_e = 1.1$, $r_d = 0.2$, $m_{ee} = 1.0/1.2$, $m_{dd} = 1.0$, $m_{ed} = 0.8$, $m_{de} = 0.7$, $\mu_e = 0.01$, $\mu_d = 0.025$. Each point at which the nullclines intersect represents an equilibrium. We can see that for this choice of parameters Equation \eqref{f} has only two non-negative equilibria, while \eqref{g} has four.}
\label{feq}
\end{figure}

For our system \eqref{elliott} we can easily see that there only exist two constant equilibria by plotting the nullclines of \eqref{f},
	\begin{align}
		r_e n_e (1 - m_{ee} n_e - m_{ed} n_d) -\mu_e n_e + \mu_d n_d &= 0,\label{nulle1}\\
		r_d n_d (1 - m_{de} n_e - m_{dd} n_d) -\mu_d n_d + \mu_e n_e &= 0,\label{nulld1}
	\end{align}
and observe where they intersect. The nullclines confirm for a specific choice of parameters that we only have two non-negative equilibria and that they are the extinction equilibrium and a single coexistence equilibrium (Figure \ref{feq} (i)). The nullclines appear as they do in Figure \ref{feq} (i) if the parameters satisfy the conditions
	\begin{equation}\label{asymproot}
		\frac{\mu_d}{r_e m_{ed}} < \frac{r_e - \mu_e}{r_e m_{ee}}, \,\,\,\,\,\,\,\, \frac{\mu_e}{r_d m_{de}} < \frac{r_d - \mu_d}{r_d m_{dd}},
	\end{equation}
which aligns with our assumption that the mutation is relatively small. We note that it is clearly also possible to deal with cases in which the mutation does not satisfy assumptions \eqref{asymproot}, however we are only interested in the case of small mutation here.

However, simply plotting the nullclines of our system does not tell us the stability of each equilibrium, we therefore consider a modified version of \eqref{f} without the mutation terms, which we call $g$,
\begin{equation}\label{g}
		g(n_e,n_d) = 
			\left( \begin{array}{c}
				r_e n_e (1 - m_{ee} n_e - m_{ed} n_d)\\
				r_d n_d (1 - m_{de} n_e - m_{dd} n_d)
			\end{array} \right).
\end{equation}
Due to the relative smallness of the mutation terms we can then introduce them as a perturbation before using the implicit function theorem.

First we evaluate the equilibria of $g$. We can easily see that there are four equilibria by plotting the nullclines,
	\begin{align}
		r_e n_e (1 - m_{ee} n_e - m_{ed} n_d) &= 0,\label{nulle2}\\
		r_d n_d (1 - m_{de} n_e - m_{dd} n_d) &= 0,\label{nulld2}
	\end{align}
which we do in Figure \ref{feq} (ii) for certain parameters. The equilibria of \eqref{g} consist of an extinction equilibrium $(0,0)$, two equilibria on the axes where one phenotype is present while the other is extinct, $\left(1/m_{ee},0\right)$, $\left(0,1/m_{dd}\right)$, and a coexistence equilibrium
	\begin{equation*}\label{coex}
		\left( \dfrac{m_{dd}-m_{ed}}{m_{ee}m_{dd}-m_{ed}m_{de}},\dfrac{m_{ee}-m_{de}}{m_{ee}m_{dd}-m_{ed}m_{de}} \right)
	\end{equation*}
which we refer to as $(n^*_e, n^*_d)$ for simplicity. Note that $(n^*_e, n^*_d)$ is a coexistence equilibria due to the condition \eqref{compe} specified earlier.

The Jacobian of \eqref{g} is
	\begin{equation}\label{gjacog}
		J_g(n_e, n_d) = 
			\left( \begin{array}{cc}
				r_e(1 - 2 m_{ee} n_e - m_{ed} n_d) &  - r_e m_{ed} n_e\\
				- r_d m_{de} n_d & r_d(1 - m_{de} n_e -2 m_{dd} n_d)
			\end{array} \right).
	\end{equation}
Substituting in values of $n_e$ and $n_d$ at each of the equilibria to the trace and determinant of \eqref{gjacog} we see that the equilibrium $(n_e^*, n_d^*)$ is stable, while the other three are unstable. Note also that the determinant of \eqref{gjacog} is non-zero when evaluated at each of the equilibria of \eqref{g}.

We now use the implicit function theorem \cite{apostol} to determine how each equilibrium moves when mutation is introduced to the system \eqref{g} as a perturbation. To do so, we suppose that there exists $\mu > 0$ such that
	\begin{equation}\label{ift}
		f(n_e,n_d) = g(n_e,n_d) +\mu M \left( \begin{array}{c} n_e\\n_d \end{array} \right)
	\end{equation}
where $g$ is defined in \eqref{g} above, $\mu$ is a non-negative scalar parameter which we use to vary the mutation and $M$ is the matrix of mutation coefficients
	\begin{equation}
		M =
			\left(\begin{array}{c c}
				-e & d\\
				e & -d
			\end{array}\right).
	\end{equation}
Clearly \eqref{ift} is a special case of \eqref{f} with $\mu_e = \mu e$ and $\mu_d = \mu d$.

The equilibria for our original system \eqref{elliott} satisfy $f(n_e,n_d)=0$, where $f$ is the nonlinearity \eqref{f}, so that
	\begin{equation}\label{ift2}
		g(n_e,n_d) +\mu M \left( \begin{array}{c} n_e\\n_d \end{array} \right) = 0.
	\end{equation}
As a consequence of the implicit function theorem, in a neighbourhood of $\mu = 0$ and an equilibrium $(\bar{n}_e, \bar{n}_d)$ of $g$, there is a unique solution of \eqref{ift2} which is a continuously differentiable function of $\mu$, say $h_{(\bar{n}_e,\bar{n}_d)}(\mu)$. We can thus differentiate \eqref{ift2} in order to obtain an expression describing how an equilibrium $(\bar{n}_e, \bar{n}_d)^T$ is perturbed upon the introduction of mutation $\mu$. Since the determinant of the Jacobian matrix $J_g$ is not equal to zero at any of the equilibria, we may invert $J_g$ and obtain the expression
	\begin{equation}\label{bigt}
		\Theta (\bar{n}_e, \bar{n}_d) := \frac{\mathrm{d}}{\mathrm{d}\mu}h_{(\bar{n}_e,\bar{n}_d)}(\mu)\bigg\rvert_{\mu=0}  = - 		J_g(\bar{n}_e,\bar{n}_d)^{-1} M \left( \begin{array}{c} \bar{n}_e\\ \bar{n}_d \end{array} \right).
	\end{equation}
Clearly the extinction equilibrium $(0,0)$ remains at $(0,0)$, and the implicit function theorem ensures the local uniqueness of this equilibrium for small $\mu > 0$. Evaluating \eqref{bigt} at each of the other equilibria of $g$, we see that the equilibrium $\left(1/m_{ee},0\right)$ is perturbed into the lower right quadrant, because
	\begin{equation}
		\Theta \left(\frac{1}{m_{ee}},0\right)
		= \frac{\mu_e}{r_e r_d (m_{ee}-m_{de})} \left(\frac{r_e m_{ed}}{m_{ee}} - \frac{r_d}{m_{ee}}(m_{ee}-m_{de}), -r_e \right)^T.
	\end{equation}
Note that the term $(m_{ee}-m_{de})$ is positive due to the condition \eqref{compe}. Similarly the equilibrium $\left(0,1/m_{dd}\right)$ is perturbed into the upper left quadrant, since
	\begin{equation}
		\Theta \left(0,\frac{1}{m_{dd}}\right)
		= \frac{\mu_d}{r_e r_d (m_{dd}-m_{ed})} \left(-r_d, \frac{r_d m_{de}}{m_{dd}} - \frac{r_e}{m_{dd}}(m_{dd}-m_{ed}) \right)^T.
	\end{equation}
Finally, the coexistence equilibrium is perturbed a small amount in a direction which is dependant on the parameters of the system,
	\begin{equation}
		\Theta \left(n_e^*,n_d^*\right) =
			\left(\begin{array}{c}
				r_e n^*_e n^*_d \left[\mu_e  (m_{dd}-m_{ed})- \mu_d(m_{ee}-m_{de})\right]\\
				r_d n^*_e n^*_d \left[\mu_d  (m_{ee}-m_{dd})- \mu_e(m_{dd}-m_{ed})\right]
			\end{array}\right).
	\end{equation}

Moreover, since the Jacobian is a continuous function of $\mu$, we know that for small $\mu \neq 0$, the stability of each of the equilibria remains the same as when $\mu=0$. Therefore by introducing a small amount of mutation to our system we are left with two non-negative equilibria: an unstable extinction state $(0,0)$ and a stable coexistence state $(n^*_e, n^*_d)$.

\section{Proof of linear determinacy}

Before we apply Theorem \ref{wang}, we have some preparation to do in showing that the system \eqref{elliott} satisfies Hypotheses 1-3 for a certain choice of parameters. Firstly we construct an invariant set for \eqref{general} with the nonlinearity $f$ in \eqref{f} by using the comparison principle, Theorem \ref{comparison}. Suppose that there exist $f^+$ and $f^-$ such that $f^-(u) \leq f(u) \leq f^+(u)$ for $u \in \mathcal{C}_{k^+}$, and that $f^+$ and $f^-$ have positive equilibria at $k^+$ and $k^-$ respectively, as well as an equilibrium at $0$. Let $u(x,t)$ be a solution of \eqref{elliott} with initial condition $u_0 \in \mathcal{C}_{k^+}$. Given these conditions, the following inequalities hold,
	\begin{align}
		k^+_t - A k^+_{xx} - f^+ (k^+) = 0 = u_t - A u_{xx} - f(u) \geq u_t - A u_{xx} - f^+ (u),\\
		0 - A 0 - f^- (0) = 0 = u_t - A u_{xx} - f(u) \leq u_t - A u_{xx} - f^- (u).
	\end{align}
The comparison principle, Theorem \ref{comparisonPrin}, therefore implies that
	\begin{equation}
		0 \leq u(x,t) \leq k^+, \,\,\,\, \mathrm{for}\,\, x \in \mathbb{R},\,\,t>0,
	\end{equation}
which says that $\mathcal{C}_{k^+}$ an invariant set of the system \eqref{elliott}.

Hypothesis 2 requires that the first diagonal block of the matrix $H_{\beta} = \beta A + \beta^{-1}f'(0)$ has a Perron-Frobenius eigenvalue that is strictly larger than the Perron-Frobenius eigenvalue of all other irreducible diagonal blocks. For the system \eqref{elliott} we have
\begin{equation}
H_\beta =
\left(
\begin{array}{c c}
\beta D_e + \beta^{-1}(r_e - \mu_e) & \beta^{-1}\mu_d \\
\beta^{-1}\mu_e & \beta D_d + \beta^{-1}(r_d - \mu_d)
\end{array}
\right).
\end{equation}
The off-diagonal elements of $H_\beta$ are strictly positive, and thus $H_\beta$ is itself irreducible. This matrix always has at least one positive eigenvalue for all values of $\mu_e$ and $\mu_d$ \cite{aled}. Hypothesis 2 is therefore satisfied due to the existence of only one irreducible diagonal block, and the matrix $H_\beta$ has a Perron-Frobenius eigenvalue which is strictly positive.

We will now describe a possible choice for $f^+$ and $f^-$ satisfying Hypothesis 1 and 3 which will show that we can apply Theorem \ref{wang}, and hence that the system of equations \eqref{elliott} is linearly determinate.

\subsection{Upper bound}

We begin by looking for an upper bound $f^+$ which is also cooperative. If we examine the Jacobian matrix \eqref{Jf} of $f$, we see that it is the terms containing inter-species competition which cause the system to become non-cooperative as the population density increases. Therefore we consider a function $f^+$ which is similar to $f$ but with $m_{ed} = m_{de} = 0$, namely
\begin{equation}\label{fpl}
f^+(n_e,n_d) = 
\left( \begin{array}{c}
r_e n_e (1 - m_{ee} n_e) - \mu_e n_e + \mu_d n_d \\
r_d n_d (1 - m_{dd} n_d) + \mu_e n_e - \mu_d n_d
\end{array} \right).
\end{equation}
We now check that this function satisfies Hypotheses 1 and 3.

\begin{proposition}
Suppose that $\mu_e < r_e$ and $\mu_d < r_d$. Then the function $f^+$ in \eqref{fpl} satisfies Hypotheses 1 and 3.
\end{proposition}

\begin{proof}
The function $f^+$ bounds $f$ from above for any $n_e \geq 0$, $n_d \geq 0$. We verify that this new system is cooperative by checking that the off-diagonal terms of the Jacobian matrix,
\begin{equation}\label{Jfplus}
J_{f^+}(n_e, n_d) = 
\left( \begin{array}{cc}
r_e(1-2 m_{ee} n_e)-\mu_e &  \mu_d \\
\mu_e & r_d(1-2 m_{dd} n_d)-\mu_d
\end{array} \right) 
\end{equation}
are strictly positive, which is true since $\mu_e, \mu_d > 0$. Evaluating \eqref{Jfplus} at $(0,0)$ we obtain,
\begin{equation}\label{Jfplus0}
J_{f^+}(0,0) = 
\left( \begin{array}{cc}
r_e-\mu_e &  \mu_d \\
\mu_e & r_d-\mu_d
\end{array} \right),
\end{equation}
which, using \eqref{Jf}, is clearly the same as the Jacobian at $(0,0)$ for $f$.

\begin{figure}[ht]
\begin{center}
\includegraphics[scale=1]{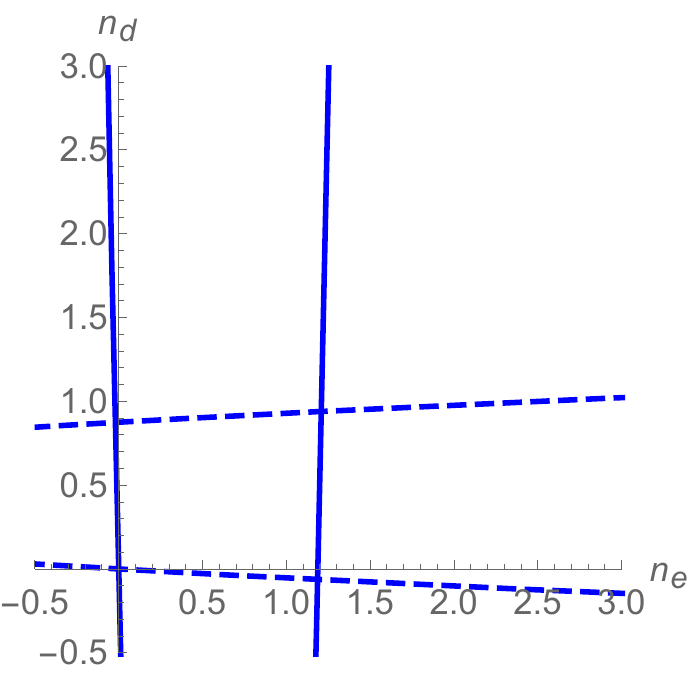}
\end{center}
\caption{Nullclines of $f^+(n_e,n_d)$ with parameter values $D_e = 0.3$, $D_d = 1.5$, $r_e = 1.1$, $r_d = 0.2$, $\mu_e = 0.01$, $\mu_d = 0.025$, $m_{ee} = 1.0/1.2$, $m_{dd} = 1.0$.}
\label{FPLUS}
\end{figure}

It is also required that $f^+$ has two non-negative equilibria, $0$ and $k^+$, with no other positive equilibria lying between the two. For this purpose consider the nullclines of function $f^+$,
\begin{align}
r_e n_e (1 - m_{ee} n_e) -\mu_e n_e + \mu_d n_d &= 0,\label{nulle3}\\
r_d n_d (1 - m_{dd} n_d) -\mu_d n_d + \mu_e n_e &= 0.\label{nulld3}
\end{align}
Both \eqref{nulle3} and \eqref{nulld3} define one density as a quadratic function of the other density. Both of these functions have one root at zero and if we impose the conditions $r_e > \mu_e$ and $r_d > \mu_d$ then the remaining root of each quadratic function will be positive, thus ensuring that they intersect in the positive quadrant and the existence of a unique positive equilibrium $k^+$.  This fits with our biologically realistic assumption that mutation rate is generally small in organisms. Figure \ref{FPLUS} illustrates the nullclines of $f^+$ for a specific choice of parameters.

We further require that the equilibrium $k^+$ of $f^+$ bounds the equilibrium $k$ of $f$ from above. The positive roots of the curves in Figure \ref{FPLUS} are
\begin{equation}\label{nullclroots}
\frac{r_e - \mu_e}{r_e m_{ee}}, \,\,\,\, \frac{r_d - \mu_d}{r_d m_{dd}},
\end{equation}
respectively. These are the same as the positive roots of the quadratic functions defined by \eqref{nulle1} and \eqref{nulld1}, which we saw earlier on the right-hand sides of the inequalities in \eqref{asymproot}. We know that the roots \eqref{nullclroots} occur at a lower population density than that of the coexistence equilibrium,
\begin{equation}\label{equkp}
\left(\frac{r_e - \mu_e}{r_e m_{ee}}, \frac{r_d - \mu_d}{r_d m_{dd}}\right) < (k^+_e,k^+_d) =: k^+,
\end{equation}
because of the quadratic nature of $f^+$, and we can see this in  Figure \ref{FPLUS}. However when we consider $f$ we see the opposite,
\begin{equation}\label{equk}
\left(\frac{r_e - \mu_e}{r_e m_{ee}}, \frac{r_d - \mu_d}{r_d m_{dd}}\right) > (k_e, k_d ) \,\,\,\, =: k.
\end{equation}
It follows from \eqref{equkp} and \eqref{equk} that
\begin{equation}
k < \left(\frac{r_e - \mu_e}{r_e m_{ee}}, \frac{r_d - \mu_d}{r_d m_{dd}}\right) < k^+,
\end{equation}
thus proving $k^+$ bounds $k$ from above.

For Hypothesis 3 we need to show that inequality \eqref{h3} holds for $f^+$ and a Perron-Frobenius eigenvector $\nu_\lambda$ of $f'(0)$. In fact with our chosen $f^+$, we can show that this inequality holds true for any positive vector $(n_e,n_d)^T$, with $n_e > 0$, $n_d > 0$, since
\begin{equation}
\alpha \left( \begin{array}{c}
r_e n_e (1- \alpha m_{ee} n_e) - \mu_e n_e + \mu_d n_d \\
r_d n_d (1- \alpha m_{dd} n_d) - \mu_d n_d + \mu_e n_e
\end{array}\right) \leq
\alpha \left( \begin{array}{c}
r_e n_e - \mu_e n_e + \mu_d n_d \\
r_d n_d - \mu_d n_d + \mu_e n_e
\end{array}\right).
\end{equation}
\end{proof}

\subsection{Lower bound}

We now require a function $f^-$ which bounds $f$ from below while satisfying Hypotheses 1 and 3. While finding a function $f^+$ was simply a case of setting some parameters equal to zero, we require a little more work to find a suitable choice of $f^-$. Ideally we need a function which behaves like $f$ at low population densities but does not have negative off diagonal elements in its Jacobian at higher population densities. We achieve this by replacing one of the population densities in the inter-species competition term with the functions $h_e$, $h_d$, so that
\begin{equation}\label{fmini}
f^-(n_e,n_d) = 
\left( \begin{array}{c}
r_e n_e (1 - m_{ee} n_e - m_{ed} h_d(n_e,n_d)) - \mu_e n_e + \mu_d n_d \\
r_d n_d (1 - m_{de} h_e(n_e,n_d) - m_{dd} n_d) + \mu_e n_e - \mu_d n_d
\end{array} \right),
\end{equation}
where
\begin{gather}
\begin{aligned}
h_d(n_e,n_d) &= \gamma_d(n_e)n_d + (1-\gamma_d(n_e))N_d\\
h_e(n_e,n_d) &= \gamma_e(n_d)n_e + (1-\gamma_e(n_d))N_e,
\end{aligned}
\end{gather}
the constant vector $(N_e, N_d)$ is an upper bound for $k^+$, and the cut-off functions $\gamma_d$, $\gamma_e \in \mathcal{C}^\infty [0,\infty)$ are defined by
\begin{gather}
\begin{aligned}
\gamma_d(n_e) &= \left\{
     \begin{array}{lr}
       1, & 0 \leq n_e \leq \tfrac{\mu_d}{4 r_e m_{ed}}\\
       \mathrm{smooth \,\, and \,\, decreasing,} & \tfrac{\mu_d}{4 r_e m_{ed}} \leq n_e \leq \tfrac{\mu_d}{2 r_e m_{ed}}\\
       0, & n_e \geq \tfrac{\mu_d}{2 r_e m_{ed}}
     \end{array}
   \right.\\
\gamma_e(n_d) &= \left\{
     \begin{array}{lr}
       1, & 0 \leq n_d \leq \tfrac{\mu_e}{4 r_d m_{de}}\\
       \mathrm{smooth \,\, and \,\, decreasing,} & \tfrac{\mu_e}{4 r_d m_{de}} \leq n_d \leq \tfrac{\mu_e}{2 r_d m_{de}}\\
       0, & n_d \geq \tfrac{\mu_e}{2 r_d m_{de}}
     \end{array}
   \right. .
\end{aligned}
\end{gather}
Thus when the population density $n_d$ is small, $h_d$ is simply $n_d$. However as $n_d$ approaches the range in which it would cause the system to become non-cooperative it is switched to a constant value $N_d > 0$. The function $h_e$ behaves similarly, switching from $n_e$ to a constant $N_e > 0$. This switching is controlled by the functions $\gamma_d$ and $\gamma_e$.

Since it is important that this $f^-$ bounds $f$ from below we need to choose the constants $N_e$ and $N_d$ to be sufficiently large. We look for upper bounds for $n_e$ and $n_d$ in the hope of using these as our choice for $N_e$ and $N_d$. Since $\mathcal{C}_{k^+}$ is an invariant set and we only consider initial data $u_0 \in \mathcal{C}_{k^+}$, it suffices to find an upper bound for $k^+$. Adding the nullclines \eqref{nulle3} and \eqref{nulld3} of $f^+$ yields
\begin{equation}
r_e n_e (1-m_{ee} n_e) + r_d n_d (1-m_{dd} n_d) = 0.
\end{equation}
From here a further rearrangement and some simple calculus gives us the inequality
\begin{equation}
r_d n_d (m_{dd} n_d - 1) = r_e n_e (1 - m_{ee} n_e) \leq \frac{r_e}{4m_{ee}}.
\end{equation}
Rearranging to obtain a quadratic inequality in $n_d$, and solving in the usual way, we obtain
\begin{equation}
n_d \leq \frac{1 + \sqrt{1+\frac{r_e m_{dd}}{r_d m_{ee}}}}{2 m_{dd}},
\end{equation}
which we use to define an upper bound for $n_d$, and similarly for $n_e$,
\begin{equation}
N_d := \frac{1 + \sqrt{1+\frac{r_e m_{dd}}{r_d m_{ee}}}}{2 m_{dd}}, \,\,\,\,\,\, N_e := \frac{1 + \sqrt{1+\frac{r_d m_{ee}}{r_e m_{dd}}}}{2 m_{ee}}.\label{bigne}
\end{equation}
Note that it only matters that $N_e$, $N_d$ bound $n_e$, $n_d$ from above respectively; \eqref{bigne} gives an example of such bounds, but we make no claims that these are sharp.

We now show that under some conditions on parameters, our choice of function $f^-$ in \eqref{fmini} satisfies Hypotheses 1 and 3.



\begin{proposition}
Suppose that
\begin{equation}
m_{ed} < \frac{1}{N_d}, \,\,\,\,\,\,\,\, m_{de} < \frac{1}{N_e},
\end{equation}
and
\begin{equation}
\begin{aligned}
\mu_e &< \min \left\{ \frac{r_e (1 - m_{ed} N_d)}{2}, \frac{r_d m_{de}(1-m_{de} N_e)}{m_{dd}} \right\},\\
\mu_d &< \min \left\{ \frac{r_d (1 - m_{de} N_e)}{2}, \frac{r_e m_{ed}(1-m_{ed}N_d)}{m_{ee}}  \right\}.
\end{aligned}
\end{equation}
Then the function $f^-$ in \eqref{fmini} satisfies Hypotheses 1 and 3.
\end{proposition}
\begin{proof}
Defining $N_d$ and $N_e$ as in \eqref{bigne} ensures that $f^-$ bounds $f$ from below whenever $(n_e,n_d) \in [0,k^+]$. The off-diagonal elements of the Jacobian matrix of $f^-$ are
\begin{equation}\label{offdiags}
\frac{\partial f^-_1}{\partial n_d} = \mu_d - r_e m_{ed} n_e \gamma_d, \,\,\,\,\,\,\,\, \frac{\partial f^-_2}{\partial n_e} = \mu_e - r_d m_{de} n_d \gamma_e.
\end{equation}
As each population density increases, the functions $\gamma_d$ and $\gamma_e$ shrink the negative terms in \eqref{offdiags} until $\gamma_d (n_e) = 0$ at $n_e =  \mu_d / (2 r_e m_{ed})$, and $\gamma_e (n_d) = 0$ at $n_d = \mu_e / (2 r_d m_{de})$. Since the off-diagonal elements would normally become negative at $n_e =  \mu_d / (r_e m_{ed})$ and $n_d = \mu_e / (r_d m_{de})$ this ensures that they remain positive for all $(n_e, n_d)$. Thus the system of equations for $f^-$ is cooperative. For population densities such that $n_e <  \mu_d / (4 r_e m_{ed})$ and $n_d < \mu_e / (4 r_d m_{de})$ we have $\gamma_d (n_e) = \gamma_e (n_d) = 1$, and thus the Jacobian of $f^-$ is identical to \eqref{Jf}. Hence, in particular, both $f$ and $f^-$ have the same Jacobian at zero.

\begin{figure}[ht]
\begin{center}
\includegraphics[scale=1]{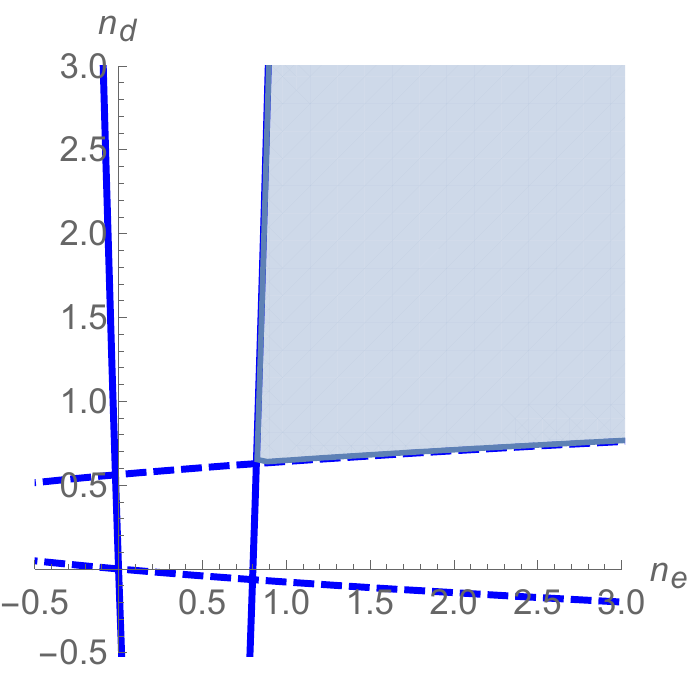}
\caption{Nullclines of $f^-_{\star}(n_e,n_d)$ with parameter values  $D_e = 0.3$, $D_d = 1.5$, $r_e = 1.1$, $r_d = 0.2$, $\mu_e = 0.01$, $\mu_d = 0.025$, $m_{ee} = 1.0/1.2$, $m_{dd} = 1.0$, $m_{ed} = 0.8$, $m_{de} = 0.7$}
\label{FMIN}
\end{center}
\end{figure}

If we replace $h_d$ and $h_e$ in $f^-(n_e, n_d)$ with $N_d$ and $N_e$ respectively, we obtain the function
\begin{equation}\label{fstar}
f^-_{\star}(n_e,n_d) = 
\left( \begin{array}{c}
r_e n_e (1 - m_{ee} n_e - m_{ed} N_d) - \mu_e n_e + \mu_d n_d \\
r_d n_d (1 - m_{de} N_e - m_{dd} n_d) + \mu_e n_e - \mu_d n_d
\end{array} \right),
\end{equation}
the nullclines of which we plot in Figure \ref{FMIN}. A coexistence equilibrium of \eqref{fstar} exists if we impose the conditions
\begin{equation}\label{condit}
m_{ed} < \frac{r_e - \mu_e}{r_e N_d}, \,\,\,\,\,\,\,\, m_{de} < \frac{r_d - \mu_d}{r_d N_e}
\end{equation}
on the smallness of the inter-morph competition. However we do not know a priori if $f^-$ has the form $f^-_\star$ for the range of values $(n_e, n_d)$ where this coexistence equilibrium occurs. This is also a coexistence equilibrium of $f^-$ if the non-zero roots of the quadratic functions defined by each of the nullclines of $f^-_\star$ are greater than the points at which $f^-$ switches to $f^-_\star$, which in the case of the first component means that,
\begin{equation}\label{rootswitch}
\frac{r_e - \mu_e - r_e m_{ed} N_d}{r_e m_{ee}} > \frac{\mu_d}{2 r_e m_{ed}}.
\end{equation}
Since the right hand side of \eqref{rootswitch} consists only of positive constants, we require that the numerator of the left hand side is positive, and thus \eqref{condit} is implied by \eqref{rootswitch}. Rearranging \eqref{rootswitch} gives the equivalent condition,
\begin{equation}\label{recondsw}
r_e - \mu_e - r_e m_{ed} N_d - \frac{m_{ee} \mu_d}{2 m_{ed}} > 0,
\end{equation}
which is satisfied, for instance, by imposing the conditions
\begin{equation}\label{cond2}
m_{ed} < \frac{1}{N_d}, \,\,\,\, \mu_e < \frac{r_e (1 - m_{ed} N_d)}{2}, \,\,\,\, \mu_d < \frac{r_e m_{ed}(1-m_{ed}N_d)}{m_{ee}}.
\end{equation}
Similarly, from the second component of \eqref{fstar} we obtain the conditions
\begin{equation}\label{cond3}
m_{de} < \frac{1}{N_e}, \,\,\,\, \mu_d < \frac{r_d (1 - m_{de} N_e)}{2}, \,\,\,\, \mu_e < \frac{r_d m_{de}(1-m_{de} N_e)}{m_{dd}}.
\end{equation}
Thus provided that the inter-morph competition satisfies
\begin{equation}\label{eqconda}
m_{ed} < \frac{1}{N_d}, \,\,\,\, m_{de} < \frac{1}{N_e},
\end{equation}
and that the mutation satisfies
\begin{equation}\label{eqcondb}
\begin{aligned}
\mu_e &< \min \left\{ \frac{r_e (1 - m_{ed} N_d)}{2}, \frac{r_d m_{de}(1-m_{de} N_e)}{m_{dd}} \right\},\\
\mu_d &< \min \left\{ \frac{r_d (1 - m_{de} N_e)}{2}, \frac{r_e m_{ed}(1-m_{ed}N_d)}{m_{ee}}  \right\},
\end{aligned}
\end{equation}
the coexistence equilibrium of $f^-_\star$ is also a coexistence equilibrium of $f^-$.  It now remains to be shown that there exists no other coexistence equilibrium for $f^-$. In Figure \ref{FMIN} we plot the nullclines of $f^-_\star$,
\begin{align}
r_e n_e (1 - m_{ee} n_e - m_{ed} N_d) - \mu_e n_e + \mu_d n_d &= 0, \label{nullstar1}\\
r_d n_d (1 - m_{de} N_e - m_{dd} n_d) + \mu_e n_e - \mu_d n_d &= 0, \label{nullstar2}
\end{align}
for parameters satisfying conditions \eqref{eqconda} and \eqref{eqcondb}. Since \eqref{nullstar1} defines $n_d$ as a quadratic function of $n_e$ we know that $f^-_{\star_1} > 0$ in the region above this curve. Similarly, to the right of the curve \eqref{nullstar2} we have $f^-_{\star_2} > 0$. We therefore must have that any other possible coexistence equilibria of $f^-$ must lie in the shaded area of Figure \ref{FMIN} since $f^-_\star < f^-$. However, for this range of population densities, we know that $f^-$ is the same as $f^-_\star$, and thus no further coexistence equilibria can occur for $f^-$. Since $f^-_\star \leq f^- \leq f$ for $(n_e, n_d) \in [0, k^+]$ we know that any coexistence equilibria of $f$ must also be situated in the shaded region of Figure \ref{FMIN}, and thus $k^- < k$. Hypothesis 3 for $f^-$ simply follows from the fact that $f^- \leq f \leq f^+$ whenever $(n_e,n_d)\in [0,k^+]$.
\end{proof}

We now gather our assumptions and summarise our result with the following theorem.

\begin{theorem}\label{ldsummary}
Suppose that
\begin{equation}\label{intersmall}
m_{ed} < \frac{1}{N_d}, \,\,\,\,\,\,\,\, m_{de} < \frac{1}{N_e},
\end{equation}
and
\begin{equation}\label{musmall}
\begin{aligned}
\mu_e &< \min \left\{ \frac{r_e (1 - m_{ed} N_d)}{2}, \frac{r_d m_{de}(1-m_{de} N_e)}{m_{dd}} \right\},\\
\mu_d &< \min \left\{ \frac{r_d (1 - m_{de} N_e)}{2}, \frac{r_e m_{ed}(1-m_{ed}N_d)}{m_{ee}}  \right\}.
\end{aligned}
\end{equation}
Then system \eqref{general} with $A$ and $f$ given by \eqref{A} and \eqref{f} respectively, satisfies Theorem \ref{wang}.
\end{theorem}

Theorem \ref{ldsummary} tells us that under the biologically realistic condition that the mutation and inter-morph competition are sufficiently small, Theorem \ref{wang} is satisfied and yields the existence of a minimal travelling wave speed $c^*$, the existence of travelling waves for $c \geq c^*$, as well as linear determinacy of $\eqref{elliott}$. While conditions \eqref{intersmall} and \eqref{musmall} are sufficient for linear determinacy, we in fact only require that \eqref{rootswitch} holds. Note that we neither claim, nor expect, that these conditions are sharp; they ensure linear determinacy in the ecologically important and biologically realistic small mutation rate case using the intuitive method of trapping $f$ between $f^-$ and $f^+$.

\section{Consequences of linear determinacy}

We now use our result on linear determinacy to investigate the behaviour of travelling wave solutions as a function of the mutation parameter $\mu$ when $f$ has the special form \eqref{ift}. This is biologically very interesting as, whilst mutation rates are kept low across organisms by natural selection, the rate can still vary by order of magnitude \cite{lynch}. We will give results on the speed of minimal-speed travelling waves as a function of mutation, as well as the composition of the leading edge of minimal-speed travelling-wave solutions of the linearised system \eqref{lin0}.

Note that Theorem \ref{wang} only rigorously establishes that $\lim_{\xi \rightarrow \infty} w (\xi)=0$ when $w(x-ct)$ is a non-minimal speed travelling-wave solution of \eqref{elliott}, and that solutions $w>0$ of \eqref{ode2o} with $\lim_{\xi \rightarrow \infty} w(\xi) = 0$ are related to the Perron-Frobenius eigenvector $q_\beta$ and corresponding eigenvalue $c$ of the matrix $H_\beta$ defined in \eqref{hlambda}\cite[Page 163, Lemma 2.4]{volpert3}. In fact, though, as mentioned in the introduction, Girardin \cite{girardin} has recently shown that $\lim_{\xi \rightarrow \infty} w (\xi)=0$ does still hold when $c=c^*$, and in \cite{girardin2} further characterises the behaviour of $w$ in terms of the Perron-Frobenius eigenvector $q_\beta$ corresponding to the minimal speed Perron-Frobenius eigenvalue $c^*$. This complements and gives additional weight to our study here of $c^*$, $\beta$ and $q_\beta$ as functions of the mutation rate $\mu$.

\subsection{Spreading speed, $c^*$}

We consider the special case of equation \eqref{ode2o} where $f(w) = g(w) + \mu M w$, namely
\begin{equation}\label{ode}
A w'' + c w' + g(w) + \mu M w = 0.
\end{equation}
In Section 3 we showed that under certain conditions on parameters, the spreading speed of \eqref{ode} is linearly determinate. Linearizing \eqref{ode} about the $(0,0)$ equilibrium and substituting the ansatz $w(\xi) = e^{- \beta \xi} q$, where $\beta > 0$ is real and $q > 0$ is a positive vector, we obtain the eigenvalue problem
\begin{equation}\label{evproblemo}
\frac{1}{\beta}(\beta^2 A +  g'(0) + \mu M)q = c q,
\end{equation}
which is a special case of \eqref{eig22}. We define the matrix $H_{\beta, \mu} := \beta A +  \beta^{-1}( g'(0) + \mu M)$. Then the eigenvalue $c$ of $H_{\beta,\mu}$ is the travelling wave speed of the profile $w(\xi) = e^{- \beta \xi} q$, while the eigenvector $q$ gives the composition of the leading edge of the travelling wave. We are interested in non-negative eigenvectors of $H_{\beta, \mu}$ because densities $n_e$ and $n_d$ are non-negative, and positive eigenvectors correspond to the case when both morphs are present together.

\begin{figure}[ht]
\begin{center}
\includegraphics[scale=1.0]{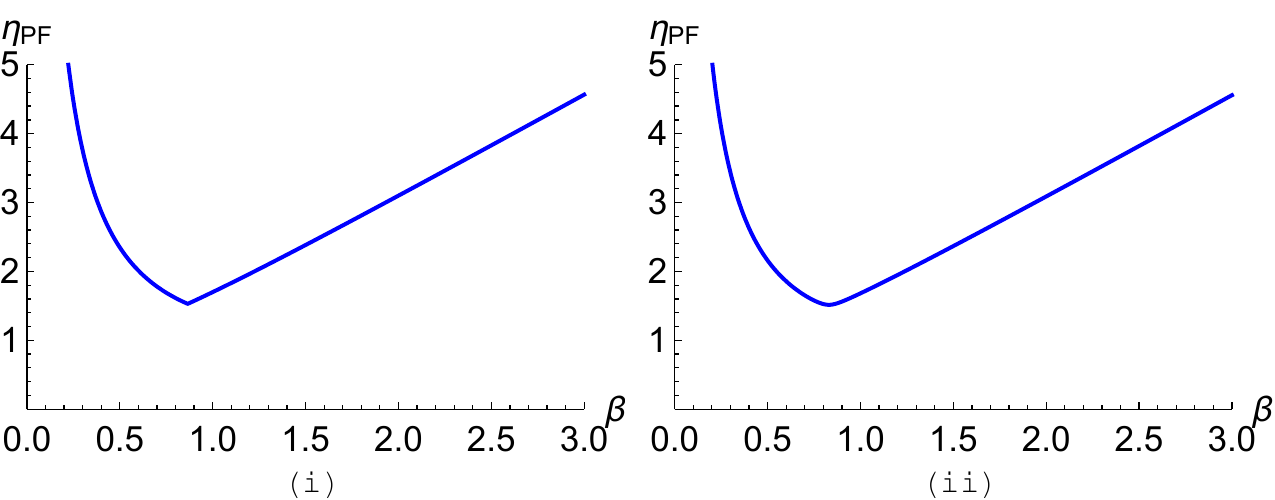}
\end{center}
\caption{(i) Perron Frobenius eigenvalue of $H_\beta$, which exists when $\mu > 0$. Parameters used are $D_e = 0.3$, $D_d = 1.5$, $r_e = 1.1$, $r_d = 0.2$, $e=0.001$, $d=0.00025$, $\mu =1$. (ii) Mutation parameter $\mu$ is increased to $\mu =100$ in order to illustrate smoothness in the $\mu > 0$ case.}
\label{egPF}
\end{figure}

For $\mu > 0$ the matrix $H_{\beta, \mu}$ has strictly positive off-diagonal elements and thus, by Theorem \ref{PerronF}, has a Perron-Frobenius eigenvalue (shown in Figure \ref{egPF}(i)), $\eta_\mathrm{PF}\left(H_{\beta, \mu}\right)$, which is the larger of the two real eigenvalues of $H_{\beta, \mu}$ and has a one-dimensional eigenspace spanned by a positive eigenvector. Of course one would expect this graph to be smooth and we note that although Figure \ref{egPF}(i) does not appear to be differentiable at the minimum, it does appear smooth on zooming in, and increasing our $\mu$ parameter we see in Figure \ref{egPF}(ii) that the minimum then does indeed appear differentiable but has a high curvature in the region of the minimum that is dependant on $\mu$. The minimal speed of the travelling-wave profile $w(\xi) = e^{- \beta \xi} q$ for a given $\mu > 0$ is given by
\begin{equation}\label{minmu}
\eta(\mu):= \inf_{\beta>0} \eta_\mathrm{PF}\left(H_{\beta, \mu}\right),
\end{equation}
the minimum value of the Perron-Frobenius eigenvalue $\eta_\mathrm{PF}\left(H_{\beta, \mu}\right)$ over $\beta>0$. We define $\beta (\mu)$ to be the value of $\beta$ at which $\eta(\mu)$ is attained.

Since calculating the minimal spreading speed involves minimizing in $\beta$ and the diffusion matrix $A$ is not a multiple of the identitiy, finding an explicit expression is not very tractable. However, by adapting an argument from \cite{crooks}, we show that $\eta(\mu)$ is a non-increasing function of the mutation parameter $\mu$, which is known to vary in biological systems. Note that this is related to the so-called ``reduction phenomenon'' discussed by Altenberg \cite{altenberg}, which roughly says that, under certain conditions, greater mixing results in lowered growth. Our result, on the other hand, states that greater mixing results in a slower speed of propagation. The following lemma ensures that the main theorem of Altenberg \cite[Theorem 6]{altenberg} holds in our setting. We can in fact, make use of \cite[Theorem 6(iii)]{altenberg} to simplify the proof of Proposition \ref{prop42}.


\begin{lemma}[Cohen, {\cite{cohen}}]\label{cohen}
Let $P, Q \in \mathbb{R}^{n \times n}$ such that $P$ is diagonal and $Q$ has positive off-diagonal elements. Then the Perron-Frobenius eigenvalue of $P+Q$ is a convex function of $P$; that is, given diagonal matrices $P_1$ and $P_2$ and $0 < \alpha < 1$,
\begin{equation}
\eta_\mathrm{PF} (\alpha P_1 + (1-\alpha)P_2 + Q) \leq \alpha \eta_\mathrm{PF} (P_1 + Q) + (1-\alpha) \eta_\mathrm{\,PF} (P_2 + Q).
\end{equation}
\end{lemma}

\begin{proposition}\label{prop42}
$\eta (\mu)$ is a non-increasing function of $\mu$.
\end{proposition}

\begin{proof}
Take $\mu_0 > 0$, define $Z$ to be the $2 \times 2$ zero matrix, and define
\begin{equation}
P:= \beta(\mu_0)^2 A - \beta(\mu_0) \eta (\mu_0) I + g'(0)
\end{equation}
Note that $\beta(\mu_0) > 0$ and
\begin{align}
\eta_\mathrm{PF}(\beta^2 A - \beta \eta(\mu_0) I + g'(0) + \mu_0 M) &\geq 0, \,\,\,\, \forall \, \beta > 0,\\
\eta_\mathrm{PF}(\beta(\mu_0)^2 - \beta(\mu_0) \eta(\mu_0) I + g'(0) + \mu_0 M) &= 0.
\end{align}
Now for any $\mu > 0$, we know that
\begin{equation}\label{pfout}
\eta_\mathrm{PF} \left(P+\mu M\right) = \mu \, \eta_\mathrm{PF} \left(\frac{1}{\mu}P + M\right),
\end{equation}
and in particular, $\eta_\mathrm{PF}(P+\mu M)$ and $\eta_\mathrm{PF} \left( \frac{1}{\mu}P + M \right)$ have the same sign. Moreover,
\begin{equation}\label{pfz}
\eta_\mathrm{PF} \left(\frac{1}{\mu_0}P + M\right) = 0.
\end{equation}
Then using \cite[Theorem 6(iii)]{altenberg}, it follows from \eqref{pfz} and  $\eta_\mathrm{PF} (M) = 0$ that,
\begin{equation}
\eta_\mathrm{PF} \left( \frac{1}{\mu}P+M \right) \leq 0, \,\,\,\,\,\,\,\, \mu > \mu_0.
\end{equation}
Using \eqref{pfout}, this implies
\begin{equation}
\eta_\mathrm{PF} \left( P+\mu M \right) \leq 0, \,\,\,\,\,\,\,\, \mu > \mu_0,
\end{equation}
and substituting in the full expression for $P$ and dividing by $\beta(\mu_0)$ then yields
\begin{equation}
\eta_\mathrm{PF} (\beta(\mu_0)A + \beta(\mu_0)^{-1}(g'(0) + \mu M)) \leq \eta(\mu_0), \,\,\,\,\,\,\,\, \mu > \mu_0.
\end{equation}
Hence
\begin{equation}\label{cresult}
\inf_{\beta > 0} \eta_\mathrm{PF}(\beta A + \beta^{-1}(g'(0) + \mu M)) \leq \eta(\mu_0).
\end{equation}
Since the expression on the left hand side of \eqref{cresult} is the definition of $\eta(\mu)$, we have
\begin{equation}
\eta(\mu) \leq \eta(\mu_0).
\end{equation}
Since this holds for any $\mu > \mu_0$ and we can choose $\mu_0$ arbitrarily small we have therefore shown that $\eta (\mu)$ is a non-increasing function of $\mu$.
\end{proof}

\subsection{Behaviour at the leading edge in the limit $\mu \to 0$}

We will study the Perron-Frobenius eigenvector $q_\beta(\mu)$ corresponding to the eigenvalue $\eta (\mu)$ in the limit as the mutation rate $\mu \to 0$ in order to determine the ratio of the phenotypes in the leading edge. We do this under conditions \eqref{parmsrD} and \eqref{fastersp} on the dispersal and growth parameters, which ensures that the faster speed $v_f$ is obtained as $\mu \to 0$ and, as we will see, results in both phenotypes being present in the leading edge.

First note that for $\mu = 0$, the matrix $H_{\beta,0}$ is diagonal, and therefore $H_{\beta,0}$ does not have a Perron-Frobenius eigenvalue. However the two eigenvalues of the matrix $H_{\beta,0}$ are clearly
\begin{equation}\label{eigenvs12}
\beta D_d + \beta^{-1} r_d, \,\,\,\,\,\,\,\,\,\,\,\,\,\,\,\, \beta D_e + \beta^{-1} r_e,
\end{equation}
which we plot in Figure \ref{corner}(i), for certain parameters, as functions of $\beta$, and in Figure \ref{corner}(ii) we consider the maximum of these eigenvalues for each $\beta$ by analogy with the Perron-Frobenius eigenvalue $\eta_\mathrm{PF}(H_{\beta, \mu})$ when $\mu>0$, note the similarity here to Figure \ref{egPF}(i). The minimal travelling wave speed, which we call $\eta_0$, is obtained at the minimum of the graph in Figure \ref{corner}(ii). We denote the value of $\beta$ at which this minimal speed is obtained $\beta^*$. It can be shown that as $\mu \rightarrow 0$ it holds that $\eta(\mu) \rightarrow \eta_0$ and $\beta(\mu) \rightarrow \beta^*$ \cite{aled}.

\begin{figure}[ht]
\begin{center}
\includegraphics[scale=1.0]{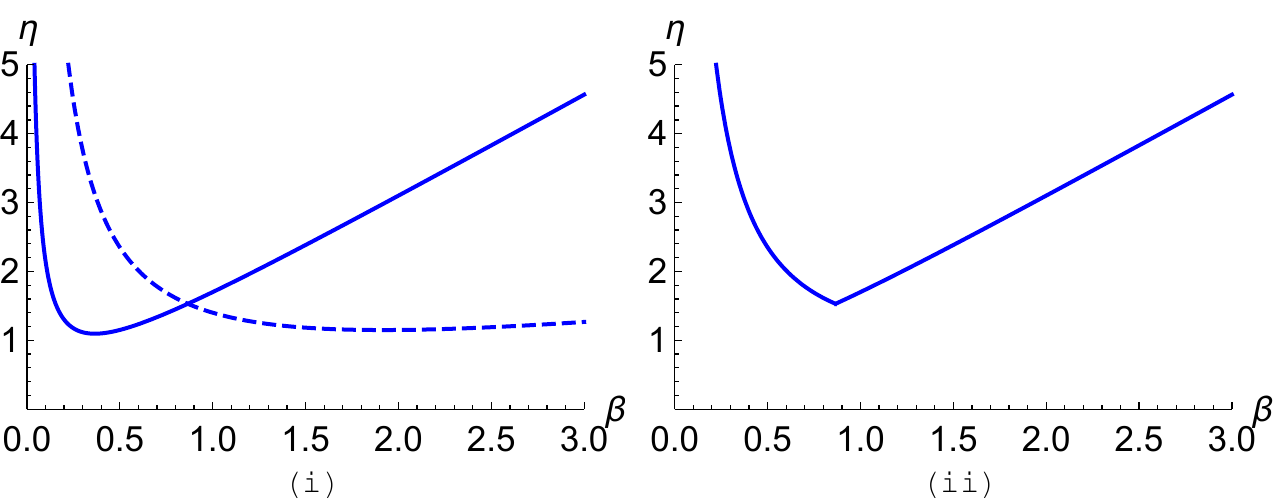}
\end{center}
\caption{(i) Eigenvalues of $H_\beta$ in the case of $\mu=0$. Parameters used are $D_e = 0.3$, $D_d = 1.5$, $r_e = 1.1$, $r_d = 0.2$. (ii) Travelling wave speed $\eta$ corresponding to each $\beta$ for $\mu=0$. The minimal travelling wave speed, $\eta_0$, is obtained at the value of $\beta$ at which both eigenvalues are equal, $\beta^*$.}
\label{corner}
\end{figure}

Note that in Figure \ref{corner} the minimum over $\beta$ of the maximum of the eigenvalues \eqref{eigenvs12} is the point at which they both meet. However there are also regions of parameters in which the eigenvalues \eqref{eigenvs12} meet in such a way that this is not the case. The condition that the minima of the two eigenvalues lie on either side of the point at which they meet, which is equivalent to the minimum value of $\beta$ of the maximum of the eigenvalues being at the crossing point, as in Figure \ref{corner}, is
\begin{equation}
\sqrt{\frac{r_d}{D_d}} < \sqrt{\frac{r_e - r_d}{D_d - D_e}} < \sqrt{\frac{r_e}{D_e}}.
\end{equation}
Closer examination of this condition reveals that it is identical to the condition \eqref{fastersp} imposed by \cite{elliott} which ensures that the faster spreading speed $v_f$ \eqref{vfspeed} is obtained in the limit as $\mu \to 0$. We henceforth write this condition in the form
\begin{equation}\label{abcondition}
a:= \beta^{*^2} D_d - r_d > 0, \,\,\,\,\,\,\,\, b := -\beta^{*^2} D_e + r_e > 0.
\end{equation}
This condition, as well as
\begin{equation}\label{rdcond2}
r_e > r_d,\,\, D_d > D_e,
\end{equation}
are assumed throughout the rest of this section. We consider this case, where $v_f$ is obtained as the limitting speed as $\mu \to 0$, to be of greatest biological interest, as unlike $v_e$ and $v_d$, $v_f$ is dependant on the traits of both morphs and occurs as a result of polymorphism.

In Figure \ref{corner}, conditions \eqref{abcondition} and \eqref{rdcond2} are satisfied, and under these conditions, the minimal travelling wave speed $\eta_0$ is obtained at the point at which the two eigenvalues of $H_{\beta,0}$ meet, therefore
\begin{equation}\label{eta0betaS}
\eta_0 = \beta^* D_d + \beta^{*^{-1}} r_d = \beta^* D_e + \beta^{*^{-1}} r_e, \,\,\,\,\,\,\,\,\,\,\,\, \beta^* = \sqrt{\frac{r_e - r_d}{D_d - D_e}}.
\end{equation}
Further, when condition \eqref{abcondition} is satisfied, $H_{\beta^* ,0}$ is a multiple of the identity it has a repeated eigenvalue with a two-dimensional eigenspace.\bigskip

We now investigate the Perron-Frobenius eigenvector $q$ of $H_\mathrm{\beta (\mu), \mu}$ in the limit $\mu \rightarrow 0$ under conditions \eqref{abcondition} and \eqref{rdcond2}. As noted earlier we will only consider the region of parameters in which \eqref{abcondition} holds. Recall that $\lim_{\mu \to 0} \eta(\mu) = \eta_0$ and $\lim_{\mu \to 0} \beta(\mu) = \beta^*$, and assume that the following limits exist,
\begin{equation}\label{limexist}
q_0 = \lim_{\mu \rightarrow 0} q, \,\,\,\, \eta'(0) = \lim_{\mu \rightarrow 0} \frac{\eta (\mu) - \eta_0}{\mu}, \,\,\,\,  \beta'(0) = \lim_{\mu \rightarrow 0} \frac{\beta (\mu) - \beta^*}{\mu}.
\end{equation}
We numerically provide further justification of these assumptions in Figure \ref{mugraphcon}. Rewriting \eqref{evproblemo} as a system of scalar equations and taking $\beta = \beta(\mu)$ we obtain
\begin{align}
		\left(\beta(\mu) D_e +  \frac{(r_e - \mu e)}{\beta(\mu)} - \eta(\mu) \right) + \frac{\mu d}{\beta(\mu)} \frac{q_2}{q_1} &= 0,\label{r1}\\
		\left(\beta(\mu) D_d +  \frac{(r_d - \mu d)}{\beta(\mu)} - \eta(\mu) \right)\frac{q_2}{q_1} + \frac{\mu e}{\beta(\mu)} &= 0,\label{r2}
\end{align}
for $\mu > 0$. We add and subtract $\eta_0$ from \eqref{r1} and \eqref{r2}, and divide by $\mu \beta(\mu)^{-1}$ to obtain,
\begin{align}
\beta^* \frac{(\beta D_e + \beta^{-1}r_e)-(\beta^* D_e + \beta^{*^{-1}}r_e)}{\mu} - \beta^* \frac{\eta-\eta_0}{\mu} - e+d \frac{q_2}{q_1} = 0\label{system01}\\
\left(\beta^* \frac{(\beta D_d + \beta^{-1}r_d)-(\beta^* D_d + \beta^{*^{-1}}r_d)}{\mu} - \beta^* \frac{\eta-\eta_0}{\mu} - d\right) \frac{q_2}{q_1} + e = 0\label{system02}
\end{align}
Bearing in mind \eqref{limexist} and taking $\mu \rightarrow 0$ we obtain a system of two equations in the three unknowns $q_{0_2}/q_{0_1}$,  $\eta'(0)$ and $\beta'(0)$,
	\begin{align}
			\beta^* \left[ D_e \beta'(0) - \frac{r_e \beta'(0)}{\beta^{*^2}} - \eta'(0)  \right] - e + d \left( \frac{q_{0_2}}{q_{0_1}} \right) &= 0,\label{system1}\\
			\beta^* \left[ D_d \beta'(0) - \frac{r_d \beta'(0)}{\beta^{*^2}} - \eta'(0)  \right] \left( \frac{q_{0_2}}{q_{0_1}} \right) - d \left( \frac{q_{0_2}}{q_{0_1}} \right) + d &= 0.\label{system2}
	\end{align}
In order to obtain a third equation, recall the that eigenvalues $\lambda$ of the matrix $H_{\beta,\mu}$ satisfy $\det{(H_{\beta,\mu} - \lambda I)} = 0$, so that
	\begin{equation}\label{eigdet}
		(\lambda - \beta D_e - \beta^{-1}(r_e - \mu e))(\lambda - \beta D_d - \beta^{-1}(r_d - \mu d)) - \beta^{-2}\mu^2 e d = 0,
	\end{equation}
and further at $\beta = \beta (\mu)$, it holds that
	\begin{equation}
		\frac{\partial \lambda}{\partial \beta} = 0, \,\,\,\, \mathrm{when}\,\,\lambda=\eta(\mu).
	\end{equation}
Differentiating \eqref{eigdet} with respect to $\beta$ and setting $\beta = \beta(\mu)$, $\lambda = \eta(\mu)$ , we then multiply the result by $\beta^3$ to obtain
	\begin{align}
		(\beta^2 (D_e &+ D_d) - (r_e + r_d - \mu e - \mu d)) \beta \, \eta(\mu) - 2 \mu^2 e  d = \label{syseq3}\\
		&( \beta^2 D_e - (r_e - \mu e))( \beta^2 D_d + (r_d - \mu d)) + ( \beta^2 D_e + (r_e - \mu e)( \beta^2 D_d -  (r_d - \mu d)).\nonumber
	\end{align}
Then differentiating with respect to $\mu$ and let  $\mu \rightarrow 0$ yields a third equation relating $\beta'(0)$ and $\eta'(0)$, namely
	\begin{equation}
		2 \beta^* (b D_d - a D_e) \beta'(0) + \eta_0 (a-b) \beta'(0) + \beta^* (a-b) \eta'(0) = (bd - ae),\label{system3}
	\end{equation}
where $a$ and $b$ are as defined in \eqref{abcondition} and
\begin{equation}
\eta_0 := \lim_{\mu \rightarrow 0} \eta (\mu) = \beta^* D_d + \beta^{*^{-1}} r_d = \beta^* D_e + \beta^{*^{-1}} r_e.
\end{equation}
We thus have a system of three equations \eqref{system1}, \eqref{system2} and \eqref{system3}, which we solve. Eliminating $\eta'(0)$ and $\beta'(0)$ gives the explicit expression for $q_{0_2}/q_{0_1}$,
	\begin{equation}\label{qratio}
		q_\mathrm{ratio} := \frac{q_{0_2}}{q_{0_1}} = \sqrt{\frac{be}{ad}} = \sqrt{\frac{(r_e - \beta^{*2} D_e) e}{(\beta^{*2} D_d - r_d) d}}
	\end{equation}
Recall that this is the ratio of the two components of the eigenvector of \eqref{evproblemo} in the limit $\mu \rightarrow 0$, where $q_2$ represents the disperser and $q_1$ the establisher. This quantity therefore predicts the ratio of the two morphs present in the leading edge of the travelling wave solution to our system. Note the presence of the positive quantities $a$ and $b$ in \eqref{qratio}, which emphasizes that this only holds in the zone \eqref{fastersp} in which the faster speed $v_f$ is obtained.

\begin{figure}[ht]
\begin{center}
\includegraphics[scale=1.0]{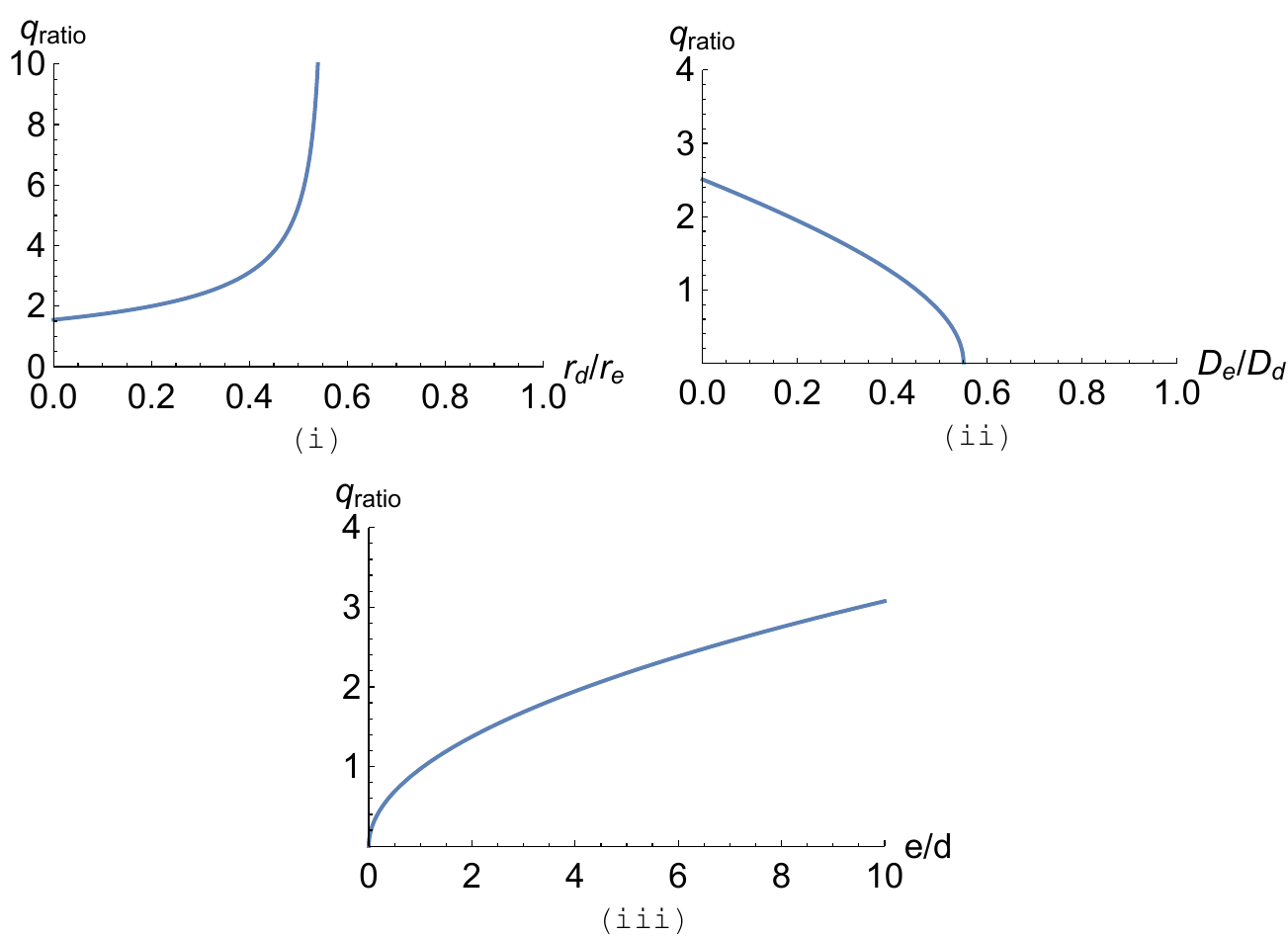}
\end{center}
\caption{Parameter sweeps of equation \eqref{qratio2}. We fix $D$ and $m$ in $(i)$, $r$ and $m$ in $(ii)$, and $r$ and $D$ in $(iii)$. When fixed, parameters take the values $r = 0.2/1.1$, $D = 0.3/1.5$, $m = 0.001/0.00025$.}
\label{rerddedd}
\end{figure}

\begin{figure}[ht]
\begin{center}
\includegraphics[scale=1.0]{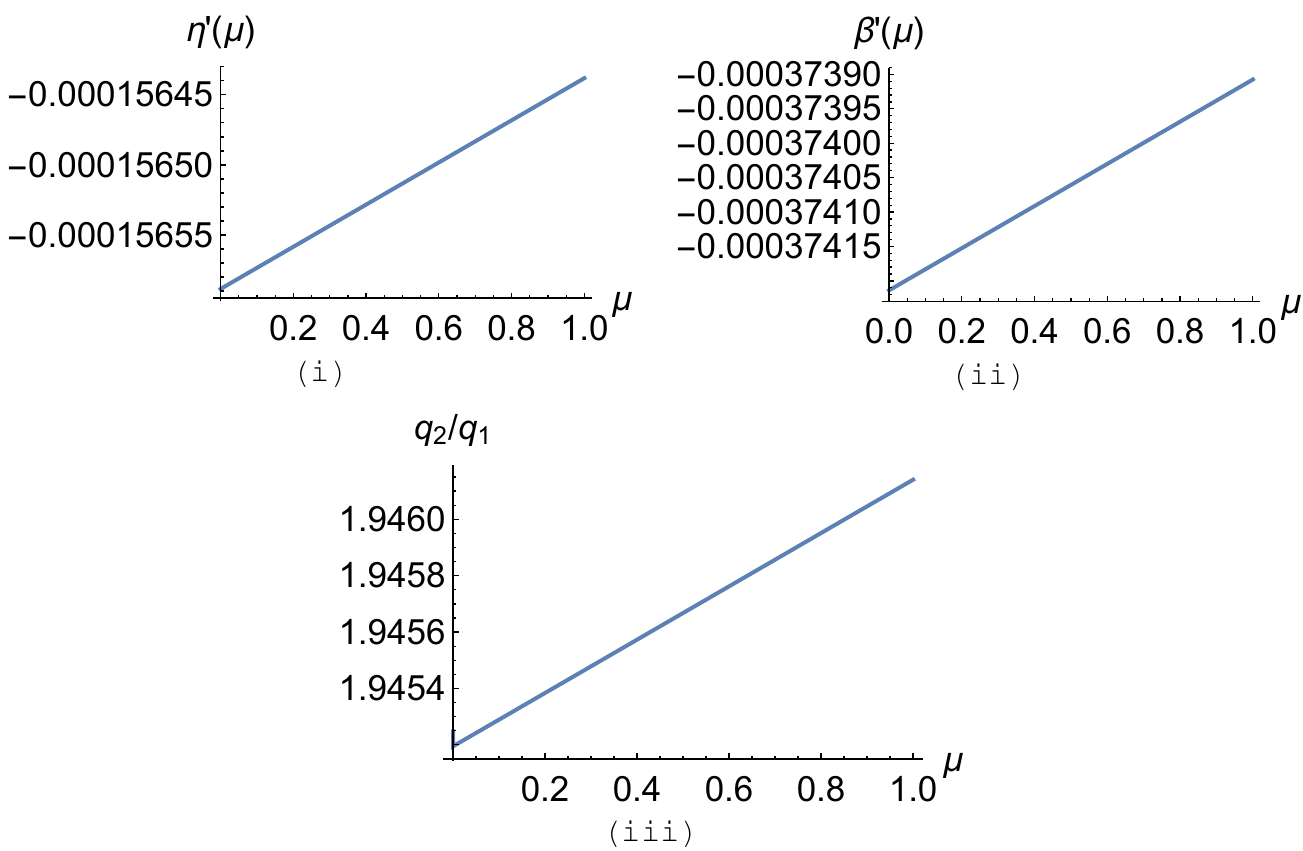}
\end{center}
\caption{Numerical solutions of $\eta'(\mu)$, $\beta'(\mu)$ and $q_2/q_1$ demonstrating convergence as $\mu \rightarrow 0$, obtained using Mathematica. Parameter values are $r_e = 1.1$, $r_d = 0.2$, $D_e = 0.3$, $D_d = 1.5$, $e = 0.001$ and $d = 0.00025$.}
\label{mugraphcon}
\end{figure}	

\begin{figure}[ht]
\begin{center}
\includegraphics[scale=1.0]{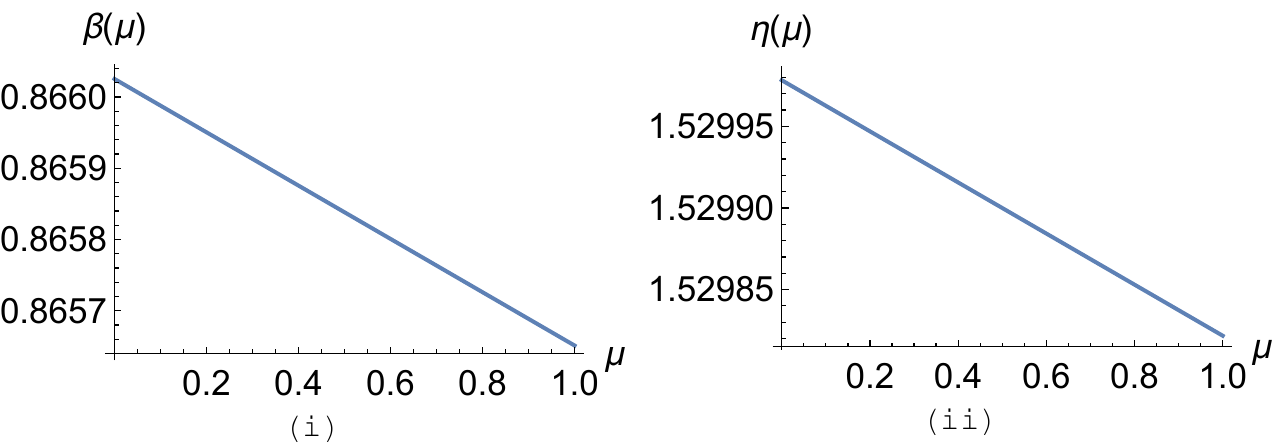}
\end{center}
\caption{Numerical solutions of $\beta(\mu)$ and $\eta(\mu)$ demonstrating convergence to $\beta^*$ and $\eta_0$, respectively, as $\mu \rightarrow 0$, obtained using Mathematica. The behaviour of $\eta(\mu)$ is consistent with our result that $\eta(\mu)$ is a non-increasing function of $\mu$. Parameter values are $r_e = 1.1$, $r_d = 0.2$, $D_e = 0.3$, $D_d = 1.5$, $e = 0.001$ and $d = 0.00025$.}
\label{mugraph}
\end{figure}

In order to investigate the effects of parameters on the proportion of each morph in the leading edge, note that if we make the substitutions $r = r_d/r_e$, $D = D_e/D_d$, $m = e/d$ we can rewrite \eqref{qratio} as a function of only three variables, the ratio of the growth rates $r$, the ratio of the dispersal rates $D$, and the ratio of the mutation rates $m$, namely
\begin{equation}\label{qratio2}
q_\mathrm{ratio} = \sqrt{\left( \frac{2D - rD - 1}{2r - rD - 1} \right) m}.
\end{equation}
This is biologically very interesting and mathematically elegant.

It is instructive and biologically relevant to comment on the effect on $q_\mathrm{ratio}$ of varying parameters $r$, $D$, and $m$. Figure \ref{rerddedd} illustrates how the ratio of growth, dispersal and mutation rates affects the proportion of each morph in the leading edge of the invasion wave. We see in Figure \ref{rerddedd} (i) that as $r_e$ increases (or $r_d$ decreases) there is an increase in the establisher morph in the leading edge. However, as $q_\mathrm{ratio}$ does not fall below $1$ there is always a larger amount of the disperser morph in the leading edge. As we increase $r_e$ to infinity (or decrease $r_d$ to 0) $q_\mathrm{ratio}$ tends to $\sqrt{(1-2/D)m}$. Note that $q_\mathrm{ratio}$ blows up when our assumption that $a>0$ and $b>0$ no longer holds. Figure \ref{rerddedd} (ii) tells us that as we increase $D_d$ (or decrease $D_e$) there is an increase in the disperser morph in the leading edge. If $D_e$ and $D_d$ are close enough in value then it is possible that the establisher morph outnumbers the disperser in the leading edge, which we see in the region in which $q_\mathrm{ratio}$ drops below one. As we increase $D_d$ to infinity (or decrease $D_e$ to 0) $q_\mathrm{ratio}$ tends to $\sqrt{m/(1-2r)}$. In this case, ratio $q_\mathrm{ratio}$ becomes zero when our assumption that $a>0$ and $b>0$ no longer holds. In Figure \ref{rerddedd} (iii) we see that if the mutation rate of a morph is increased it results in a decrease in the density of that morph in the leading edge. Since the ratio $m$ is not restricted to a particular range, like $r$ and $D$ due to \eqref{parmsrD}, we see that the mutation rate is also able to change which morph is more prevalent in the leading edge. If we fix all parameters but the mutation rates in \eqref{qratio2} we obtain $q_\mathrm{ratio} = \sqrt{k m}$, where $k > 0$ is a constant. Thus $q_\mathrm{ratio}$ goes to zero as $m$ grows smaller and blows up as $m$ tends to infinity. These provide us with interesting preditctions which could be tested experimentally with real biological systems.

Recall that we made assumptions \eqref{limexist} on the existence of $q_0$, $\eta'(0)$ and $\beta'(0)$. We finish by providing some justification for these assumptions by solving equations \eqref{r1}, \eqref{eigdet} and \eqref{syseq3} as a system of 3 equations in $q_2/q_1$, $\eta(\mu)$ and $\beta(\mu)$. Differentiating the latter two by $\mu$ we plot $\eta'(\mu)$, $\beta'(\mu)$ and $q_2/q_1$ as functions of $\mu$ in Figure \ref{mugraphcon} demonstrating the existence of $\eta'(0)$, $\beta'(0)$ and $q_\mathrm{ratio}$ for this set of parameters. Note that as $\mu$ increases we see an increase in the proportion of the phenotype $n_d$, which is expected as for this set of parameters $e > d$. Finally, we also plot $\beta(\mu)$ and $\eta(\mu)$ as functions of $\mu$ in Figure \ref{mugraph}. As $\mu \rightarrow 0$ we see $\beta(\mu)$ and $\eta(\mu)$ converge to the values $\beta^*$ and $\eta_0$, respectively. Further, as we increase $\mu$ we observe a decrease in the minimal spreading speed $\eta (\mu)$, this is consistent with our result in Section 4.1.

\section{Conclusions}

In this work we considered a Lotka-Volterra competition system with the addition of a small linear mutation rate. We first exploit the implicit function theorem to verify that the addition of a small linear mutation rate limits the system to only two steady states and prove that this system is linearly determinate under some assumptions on the smallness of the inter-morph competition and mutation parameters. We then use our result on linear determinacy in order to prove that the minimal travelling wave speed $\eta(\mu)$ is a non-increasing function of the mutation rate $\mu$, which is related to the result of Altenberg \cite{altenberg} that greater mixing slow growth. We also determine the dependence of the distribution of morphs in the leading edge on each trait in the $\mu \to 0$ limit, in the zone in which the condition \eqref{fastersp} for the faster speed $v_f$ is satisfied.

Note that in the absence of mutation, Tang and Fife \cite{tang} established the existence of travelling fronts for \eqref{elliott} for all speeds greater than or equal to $\max \{ 2 \sqrt{r_e D_e}, \, 2 \sqrt{r_d D_d}\}$. Thus, interestingly, in the case when $r_e$, $r_d$, $D_e$, $D_d$ satisfy \eqref{fastersp} and \eqref{intersmall} holds, the limit as $\mu \rightarrow 0$ of the minimal front speed $\eta (\mu)$ is strictly larger than the minimal front speed for \eqref{elliott} in the absence of mutation.

It is worth comparing and contrasting this work and that of Griette and Raoul \cite{griette}. They consider travelling wave solutions for a system similar to \eqref{elliott}, in which both morphs disperse at the same rate, $D_e = D_d$, and mutate at the same rate $\mu_e = \mu_d$. We, on the other hand, address a biologically realistic and interesting case motivated by systems in which $D_d > D_e$ and $r_e > r_d$, treat spreading speeds as well as travelling waves using an approach different from that of \cite{griette} to show linear determinacy, benefiting from \cite{wang}, and then study the existence of the linearised problem by exploiting ideas from Perron-Frobenius theory. 

Interestingly, Girardin \cite{girardin} in a recent preprint (concomitant to our study), uses different methods to investigate characteristics of spreading speeds and travelling waves that apply to systems of $N$ species related to system \eqref{elliott}. Importantly, our approach is is different and biologically more realistic and relevant, concentrating on system \eqref{elliott} when the mutation rate is small, rigorously proving claims of Elliot and Cornell \cite{elliott}, and in particular, studying in detail the ecologically important questions of how the speed and leading edge of the minimal speed travelling front solutions of the linearised problem depend on parameters in the case of small mutation rate. The latter on its own is already an important result.

A natural progression of our work would be a detailed analysis of systems for $N$ species, in particular, dependence of spreading speeds and the composition of the morphs in the leading edge on parameters. The methods developed in Section 4 have the clear potential to provide interesting results in the $N$ species case. Characterising the spreading speed as a non-increasing function of the mutation $\mu$ as we did in Section 4.1 should extend without much difficulty as we used a matrix argument which did not depend explicitly on the number of species. It would also be of interest to see if linear determinacy could be proven for models which incorporate mutation in a different manner, such as the model of Bouin et al \cite{bouin}, or a model in which the mutation coefficients themselves are a function of the population density.

\section*{Acknowledgements}

We would like to acknowledge Natalia Petrovskaya at the University of Birmingham and Sergei Petrovskii at the University of Leicester for their feedback and for organizing the META Mathematical Ecology: Theory and Applications series of events. We would also like to thank L\'{e}o Girardin at Laboratoire Jacques-Louis Lions, as well as Stephen Cornell and Vincent Keenan at the University of Liverpool for their insight and several interesting discussions. The first author would like to acknowledge Swansea University and the College of Science at Swansea University for financial support in the form of a PhD studentship.

\end{document}